\documentclass[reqno, 10pt]{amsart}
\usepackage{amssymb}
\usepackage[usenames, dvipsnames]{color}
\usepackage{stmaryrd}
\usepackage{mathrsfs}
\usepackage{esint}

\usepackage{todonotes}
\usepackage{enumerate}
\usepackage{hyperref}

\usepackage{verbatim}

\theoremstyle{plain}
\newtheorem{theorem}{Theorem}[section]
\newtheorem{lemma}[theorem]{Lemma}
\newtheorem{corollary}[theorem]{Corollary}
\newtheorem{proposition}[theorem]{Proposition}

\theoremstyle{definition}
\newtheorem{definition}[theorem]{Definition}

\newtheorem{assumption}[theorem]{Assumption}

\theoremstyle{remark}
\newtheorem{remark}[theorem]{Remark}

\numberwithin{equation}{section}

\newcommand{\bC}{\mathbb{C}}
\newcommand{\bN}{\mathbb{N}}

\newcommand{\bR}{\mathbb{R}}
\newcommand{\bH}{\mathbb{H}}
\newcommand{\bZ}{\mathbb{Z}}

\newcommand\cD{\mathcal{D}}

\newcommand\cM{\mathcal{M}}

\newcommand\cU{\mathcal{U}}

\newcommand\sL{\mathscr{L}}

\newcommand\sW{\mathscr{W}}

\newcommand\sH{\mathscr{H}}

\makeatletter
\def\dashint{\operatorname%
{\,\,\text{\bf--}\kern-.98em\DOTSI\intop\ilimits@\!\!}}
\makeatother

\begin{document}

\title[Degenerate parabolic equations]{Degenerate linear parabolic equations in divergence form on the upper half space}

\author[H. Dong]{Hongjie Dong}
\address[H. Dong]{Division of Applied Mathematics, Brown University, 182 George Street, Providence RI 02912, USA}
\email{hongjie\_dong@brown.edu}

\author[T. Phan]{Tuoc Phan}
\address[T. Phan]{Department of Mathematics, University of Tennessee, 227 Ayres Hall,
1403 Circle Drive, Knoxville, TN 37996-1320, USA}
\email{phan@math.utk.edu}

\author[H. V. Tran]{Hung Vinh Tran}
\address[H. V. Tran]{Department of Mathematics, University of Wisconsin-Madison, Van Vleck Hall
480 Lincoln Drive
Madison, WI  53706, USA}
\email{hung@math.wisc.edu}

\thanks{
H. Dong is partially supported by the NSF under agreement DMS-2055244,
the Simons Foundation, grant \# 709545, and a Simons fellowship.
T. Phan is partially supported by the Simons Foundation, grant \# 354889.
H. Tran is supported in part by NSF CAREER grant DMS-1843320 and a Simons Fellowship.
}

\begin{abstract}
We study a class of second-order degenerate linear parabolic equations in divergence form in $(-\infty, T) \times \bR^d_+$ with homogeneous Dirichlet boundary condition on $(-\infty, T) \times \partial \bR^d_+$, where $\bR^d_+ = \{x \in \bR^d\,:\, x_d>0\}$ and $T\in {(-\infty, \infty]}$ is given.
The coefficient matrices of the equations are the product of $\mu(x_d)$ and bounded uniformly elliptic matrices, where $\mu(x_d)$ behaves like $x_d^\alpha$ for some given $\alpha \in (0,2)$, which are degenerate on the boundary $\{x_d=0\}$ of the domain.
Under a partially VMO assumption on the coefficients, we obtain the wellposedness and regularity of solutions in weighted Sobolev spaces.
Our results can be readily extended to systems.
 \end{abstract}

\maketitle

 \section{Introduction}
\subsection{Setting}
Let $T\in (-\infty,\infty]$, $d\in \bN$, and $\Omega_T=(-\infty,T)\times \bR^d_+$, where $\bR^d_+ = \bR^{d-1}  \times \bR_+$ with $\bR_+ = (0, \infty)$.
Let $(a_{ij}): \Omega_T \rightarrow \mathbb{R}^{d\times d}$ be measurable and satisfy the uniform ellipticity and boundedness conditions with the ellipticity constant $\nu \in (0,1)$
\begin{equation} \label{ellipticity-cond}
\nu|\xi|^2 \leq a_{ij}(z) \xi_i \xi_j, \quad |a_{ij}(z)| \leq \nu^{-1}, \quad \forall \ z \in \Omega_T,
\end{equation}
for all $\xi = (\xi_1, \xi_2, \ldots, \xi_d) \in \bR^d$.
Also, let $c_0: \Omega_T \rightarrow \bR$ and  $\mu: \bR_+ \rightarrow \bR$ be measurable functions satisfying
\begin{equation} \label{c-mu.cond}
\nu  \leq c_0(z), \ \frac{\mu(x_d)}{x_d^\alpha}  \leq \nu^{-1}, \quad \forall \ x_d \in \bR_+, \quad \forall \ z \in \Omega_T,
\end{equation}
where $\alpha\in (0,2)$ is a fixed constant.
For $\lambda \geq 0$, let $\sL$ be the second-order linear operator  with degenerate coefficients defined by
\[
\sL u = u_t+\lambda c_0(z) u-\mu(x_d) D_i(a_{ij}(z)D_j u), \quad z= (t, x', x_d) \in \Omega_T.
\]
We study a class of equations in the form
\begin{equation}
                    \label{eq6.12}
                    \left\{
                    \begin{array}{rccl}
\sL u & = & \mu(x_d) D_i F + f &\quad\text{in}\,\,\Omega_T, \\ u & = & 0 & \quad \text{on}\,  \, (-\infty,T)\times \partial \bR^d_+.
\end{array} \right.
\end{equation}
Here, in \eqref{eq6.12},  $f :\Omega_T \rightarrow \bR$ and $F = (F_1, F_2, \ldots, F_d): \Omega_T \rightarrow \bR^d$ are given measurable functions, and $u: \Omega_T \rightarrow \bR$ is an unknown function.

It important to note that \eqref{eq6.12} has a natural scaling
$$
(t,x)\to (s^{2-\alpha}t,sx),\quad s>0.
$$
Moreover,  the PDE in \eqref{eq6.12} can be written into the following one in which the coefficients become singular on the boundary  $\{x_d =0\}$ of the domain
\begin{equation} \label{sing-PDE}
 \mu(x_d)^{-1}\big( u_t + \lambda c_0(z) u\big) - D_i\big( a_{ij} (z) D_j u + F_i)= \mu(x_d)^{-1} f \quad \text{in} \quad \Omega_T.
\end{equation}
The PDE \eqref{sing-PDE} will be used in our definition of weak solutions of \eqref{eq6.12}, in which the integration by parts is applied to the  terms  $\mu(x_d)^{-1}u_t $ and $D_i\big( a_{ij} (z) D_j u + F_i)$. Also, note that in \eqref{sing-PDE}, the coefficients $\mu(x_d)^{-1}$ and $\mu(x_d)^{-1} c_0(z)$ are not locally integrable near $\{x_d =0\}$ when $\alpha \in [1,2)$.

The aim of this paper is to show that for any $p\in (1,\infty)$,  under certain regularity assumption on $(a_{ij})$,
\begin{equation*}
\begin{split}
& \|Du\|_{L_p(\Omega_T)} +   \sqrt{\lambda} \| x_d^{-\alpha /2} u\|_{L_p(\Omega_T)}   \leq  N \big(\|F\|_{L_p(\Omega_T)}  +   \|g\|_{L_p(\Omega_T)} \big),
 \end{split}
\end{equation*}
where $N = N(\nu, d, \alpha, p)>0$, and $\lambda>0$ is sufficiently large, and $g = x_d^{1-\alpha}|f_1|+ \lambda^{-1/2} x_d^{-\alpha/2} |f_2|$ for $f = f_1 + f_2$.
See Theorem \ref{main-thrm} for the precise statements.
We also obtain a similar but more general weighted estimate (see Theorem \ref{thrm-2}).
To the best of our knowledge, this paper is the first one in which the wellposedness and regularity of solutions to the general  degenerate linear parabolic equation \eqref{eq6.12} is studied.
The above weighted $W^{1,p}$-estimate is also new in the literature.
A specific case where $\mu(x_d)=x_d$ was studied recently in our unpublished paper \cite{PhanTran}.

\subsection{Related literature}
The literature on regularity theory for degenerate elliptic and parabolic equations is vast, and we will only describe results that are related to \eqref{eq6.12}.
The H\"{o}lder regularity estimates for solutions to elliptic equations with singular and degenerate coefficients, which are $A_2$-Muckenhoupt weights, were obtained in \cite{Fabes, FKS}.
See also  the books \cite{Fichera, OR} and \cite{KoNi, MuSt1, MuSt2, Lin, Wang2, JinXiong, Sire-1, Sire-2} and the references therein for other results on the wellposedness,  H\"{o}lder, and Schauder regularity estimates for various  classes of degenerate equations.

The following equation, which is closely related to \eqref{eq6.12}, was studied much in the literature
\begin{equation} \label{lit-eqn}
 u_t(z) + \lambda u(z) - x_d \Delta u -\beta D_d u=f(z) \quad \text{in} \quad \Omega_T,
\end{equation}
where $\lambda \geq 0$ and $\beta>0$ are given constants.
Note that the requirement that $\beta>0$ is essential in the analysis of  \eqref{lit-eqn}, which is an important prototype equation appearing in the study of porous media equations and parabolic Heston equations.
The Schauder a priori estimates  in weighted H\"older spaces for solutions to \eqref{lit-eqn} and more general equations of this type were obtained in \cite{DaHa, FePo}; and the weighted $W^{2,p}$-estimates for solutions were obtained in \cite{Koch}.
Thanks to its special features, the boundary condition of \eqref{lit-eqn} on $\{x_d=0\}$ may be omitted. For us, we impose the homogenous Dirichlet boundary condition $u=0$ on $\{x_d=0\}$ in \eqref{eq6.12}, which is natural in our setting (see \cite[Theorem 2.1]{PhanTran}).
Because of the different natures of the equations, our methods and the obtained $W^{1,p}$-estimates are rather different from those in  \cite{DaHa, Koch, FePo} with different weights, and  to the best of our knowledge, they are new in the literature.

Our main motivation to study \eqref{eq6.12} comes from the analysis of degenerate viscous Hamilton-Jacobi equations.
A model equation of this kind is
\begin{equation} \label{HJB-eqn}
 u_t(z) + \lambda u(z) + H(z,Du) - x_d^\alpha \Delta u =0 \quad \text{in} \quad \Omega_T,
\end{equation}
where $H:\Omega_T \times \bR^d \to \bR$ is a given smooth Hamiltonian.
Here, $\lambda \geq 0$ and $\alpha \in (0,2)$ are given.
If $H(z,p)$ does not depend on $p$ for $(z,p) \in \Omega_T \times \bR^d$, then \eqref{lit-eqn} becomes a special case of \eqref{eq6.12}.
For typical initial-value problems of viscous Hamilton-Jacobi equations with possibly degenerate and bounded diffusions, we often have wellposedness of viscosity solutions, and such solutions are often Lipschitz in $z$ (see \cite{CIL, AT} and the references therein).
However, finer regularity of solutions is not very well understood in the literature, and in particular, optimal regularity of solutions to \eqref{HJB-eqn} near $\{x_d=0\}$ has not been investigated.
Besides, the growth of $x_d^\alpha$ in the diffusion coefficients at infinity has to be treated carefully.
As mentioned, a specific case of \eqref{eq6.12} where $\mu(x_d)=x_d$ was studied recently in our unpublished paper \cite{PhanTran}.
In this case, the equation also has a connection to the Wright-Fisher equation arising in population biology, for which the fundamental solution was studied in \cite{ChenStroock}. See also \cite{EM10} and the references therein.
We will study regularity of solutions to \eqref{HJB-eqn} and related PDEs in the future.

It is worth noting that similar results on the wellposedness and regularity estimates in weighted Sobolev spaces for equations with singular-degenerate coefficients were established in a series of papers \cite{DP20, Dong-Phan-1, Dong-Phan-2, Dong-Phan-3}.
The weights of singular/degenerate coefficients of $u_t$ and $D^2u$ in these papers appear in a balanced way, which plays a crucial role in the analysis and functional space settings.
In fact,  Harnack's inequalities were proved to be false in certain cases if the balance is lost in \cite{Chi-Se-1, Chi-Se-2}.
Of course, \eqref{eq6.12} does not have this balance structure, and our analysis is quite different from those in \cite{DP20, Dong-Phan-1, Dong-Phan-2, Dong-Phan-3}.

\subsection{Ideas of the proof}
Our proof is based on a unified kernel-free approach developed by Krylov in \cite{Krylov-07} for linear nondegenerate elliptic and parabolic equations with coefficients which are in the class of VMO with respect to the space variables and are allowed to be merely measurable in the time variable.
His proof relies on mean oscillation estimates, and uses the Hardy--Littlewood maximal function theorem and the Fefferman--Stein sharp function theorem.
A key step of the proof is to estimate the H\"older semi-norm of the derivatives of solutions to the corresponding homogeneous equations.
See \cite{MR2338417, MR2300337, Krylov, Dong-Kim11, MR3812104} and the references therein for subsequent work in this direction. Particularly,  in \cite{MR3812104} a generalized Fefferman--Stein theorem was established in weighted mixed-norm Lebesgue spaces.
The underlying space is a space of homogeneous type, which is equipped with a quasi-metric and a doubling measure.

To prove the main theorems, we construct a quasi-metric as well as a filtration of partitions (dyadic decompositions) on $\Omega_\infty$, which are suitable to \eqref{eq6.12}.
In particular, after using a proper scaling argument, they allow us to apply the interior H\"older estimates for nondegenerate equations proved in \cite{Dong-Kim11}.
The boundary H\"older estimates are more involved especially when $\alpha\in (1,2)$.
To this end, we use an energy method, the weighted Sobolev embedding, and a delicate bootstrap argument.
We consider the quantities $D_{x'} u$ and $U=a_{dj}D_ju$ instead of the full gradient $Du$.
See the proof of Proposition \ref{prop3}.
Such boundary estimates seem to be new even when the coefficients are constant. 
It is worth noting that for scalar equations with negative $\alpha$, boundary Schauder type estimates were established recently \cite{JinXiong}, which were essential in the derivation of optimal boundary regularity for fast diffusion equations.
Since we do not use the maximum principle or the DeGiorgi-Nash-Moser estimate, our results can be readily extended to the corresponding systems.

\subsection*{Organization of the paper}
The paper is organized as follows.
In Section \ref{sec-2}, we introduce the needed functional spaces, give the definition of weak solutions to \eqref{eq6.12}, and state the main results.
Preliminary analysis and $L_2$-solutions are discussed in Section \ref{sec-3}.
In Section \ref{sec-4}, we study the case when the coefficients of \eqref{eq6.12} depend only on the $x_d$-variable.
Finally, the proofs of the main results (Theorems \ref{main-thrm} and \ref{thrm-2}) are given in Section \ref{sec-5}.

\subsection*{Acknowledgement}
We would like to thank Andreas Seeger for some useful discussions.


\section{Weak solutions and main results} \label{sec-2}

\subsection{Functional spaces and definition of weak solutions}
For $p \in [1, \infty)$,  $-\infty\le S<T\le +\infty$, and for a given domain $\cD \subset \bR^d_+$, let $L_p((S,T)\times \cD)$ be  the usual Lebesgue space consisting of measurable functions $u$ on $(S,T)\times \cD$ such that the norm
\[
\|u\|_{L_p( (S,T)\times \cD)}= \left( \int_{(S,T)\times \cD} |u(t,x)|^p\, dxdt \right)^{1/p} <\infty.
\]
Also, for a given weight $\omega$ on $(S,T)\times \cD$, we define $L_{p}((S,T)\times \cD,\omega)$ to be the weighted Lebesgue space on $(S,T)\times \cD$ equipped with the norm
\begin{equation*}
\|u\|_{L_{p}((S,T)\times \cD, \omega)}=\left(\int_{S}^T\int_{\cD} |u(t,x)|^p \omega (t,x)\, dx dt\right)^{1/p}   < \infty.
\end{equation*}

Because of  the structure of \eqref{eq6.12}, the following weighted Sobolev spaces are needed.
For a fixed $\alpha \in (0,2)$ and a given weight $\widetilde \omega$ on $\cD$, we define
$$
W^1_p(\cD, \widetilde \omega)=\big\{u\,:\, ux_d^{-\alpha/2}, Du\in L_p(\cD, \widetilde \omega)\big\},
$$
which is equipped with the norm
$$
\|u\|_{W^1_p(\cD,\widetilde\omega)}=\|u x_d^{-\alpha/2}\|_{L_p(\cD, \widetilde \omega)}+\|Du\|_{L_p(\cD,\widetilde \omega)}.
$$
We note that $W^1_p(\cD, \widetilde \omega)$ depends on $\alpha$, and it is different from the usual weighted Sobolev space.

We denote by $\sW^1_p(\cD,\widetilde \omega)$ the closure in $W^1_p(\cD,\widetilde \omega)$ of all compactly supported functions in $C^\infty(\overline{\cD})$ vanishing near $\overline{\cD} \cap \{x_d=0\}$ if $\overline{\cD} \cap \{x_d=0\}$ is not empty.
The space $\sW^1_p(\cD, \widetilde \omega)$ is equipped with the same norm
$$
\|u\|_{\sW^1_p(\cD,\widetilde \omega)}=\|u\|_{W^1_p(\cD,\widetilde\omega)}.
$$
We define $W^1_p((S,T) \times \cD,\omega)$ and $\sW^1_p((S,T) \times \cD,\omega)$ in a similar way, and for $u \in  \sW^1_p((S,T) \times \cD,\omega)$,
\[
\begin{split}
\|u\|_{\sW^1_p((S,T) \times\cD, \omega)} & =\|u\|_{W^1_p((S,T) \times\cD,\omega)}\\
& =\|u x_d^{-\alpha/2}\|_{L_p((S,T) \times\cD, \omega)}+\|Du\|_{L_p((S,T) \times\cD, \omega)}.
\end{split}
\]

%
%

\noindent Next, we define
\[
\begin{split}
& \bH_{p}^{-1}( (S,T)\times \cD, \omega) \\
& =\big\{u\,:\, u  =  \mu(x_d) D_iF_i +f_1+f_2, \ \ \text{where}\ f_1 x_d^{1-\alpha},f_2x_d^{-\alpha/2}\in L_{p}( (S,T)\times \cD, \omega)\\
& \qquad\text{and }
 F= (F_1,\ldots,F_d) \in L_{p}((S,T)\times \cD, \omega)^{d}\big\},
\end{split}
\]
that is equipped with the norm
\begin{align*}
&\|u\|_{\bH_{p}^{-1}((S,T)\times \cD, \omega)} \\
&=\inf\big\{\|F\|_{L_{p}((S,T)\times \cD, \omega)}
+\||f_1x_d^{1-\alpha}|+|f_2x_d^{-\alpha/2}|\|_{L_{p}((S,T)\times \cD, \omega)}\,:\\
&\qquad u= \mu(x_d) D_iF_i +f_1+f_2\big\}.
\end{align*}
Then, we define the solution space
\[
 \sH_{p}^1((S,T)\times \cD, \omega)
 =\big\{u \,:\,  u \in  \sW^1_p((S,T) \times \cD,\omega)),
 u_t\in  \bH_{p}^{-1}( (S,T)\times \cD, \omega)\big\},
\]
which is equipped with the norm
\begin{align*}
\|u\|_{\sH_{p}^1((S,T)\times \cD, \omega)} &= \|ux_d^{-\alpha/2}\|_{L_{p}((S,T)\times \cD, \omega)} + \|Du\|_{L_{p}((S,T)\times \cD, \omega)} \\
& \qquad +\|u_t\|_{\bH_{p}^{-1}((S,T)\times \cD, \omega)}.
\end{align*}
 If $\omega \equiv 1$, we simply write $\sW^1_p((S,T) \times \cD,\omega)),  \sH_{p}^1((S,T)\times \cD, \omega)$ as $\sW^1_p((S,T) \times \cD)),  \sH_{p}^1((S,T)\times \cD)$, respectively. Now, we give the definition of weak solutions  to equation \eqref{eq6.12}.
\begin{definition}  \label{weak-sol-def}
Let $p \in (1, \infty)$, $F \in L_p((S,T)\times \cD, \omega)^d$ and $f=f_1+f_2$, where $f_1x_d^{1-\alpha},f_2x_d^{-\alpha/2} \in L_p((S,T)\times \cD,  \omega)$.
We say that $u \in \sH_{p}^1((S,T)\times \cD, \omega)$ is a weak solution to \eqref{eq6.12} in $(S,T)\times \cD$ with the boundary condition $u =0$ on $\overline{\cD} \cap \{x_d =0\}$ when $\overline{\cD} \cap \{x_d =0\} \not=\emptyset$ if
\begin{align*}
&\int_{(S,T)\times \cD} \mu(x_d)^{-1}(-u \partial_t \varphi +\lambda c_0(z)u \varphi)\, dz + \int_{(S,T)\times \cD}( a_{ij} D_{j} u  + F_i)D_{i} \varphi\, dz \\
&=   \int_{(S,T)\times \cD}
\mu(x_d)^{-1} f(z) \varphi(z)\, dz
\end{align*}
for any $\varphi \in C_0^\infty((S,T)\times \cD)$.
\end{definition}

\subsection{Balls, cylinders, and partial mean oscillations of coefficients}
For $x_0 = (x', x_{0d}) \in \bR^{d-1} \times \bR_+$ and $\rho>0$, we write $B_\rho(x_0)$ the ball in $\bR^d$ with radius $\rho$ and centered at $x_0$.
Also
\[
B_\rho^+(x_0) = B_\rho(x_0) \cap \bR^d_+,
\]
and $B_\rho'(x_0')$ is the ball in $\bR^{d-1}$ with radius $\rho$ and centered at $x_0' \in \bR^{d-1}$.

Recall that the PDE in \eqref{eq6.12} is invariant under the scaling
\[
(t,x) \mapsto (s^{2-\alpha} t, sx), \quad s > 0.
\]
Moreover, for $x_d \sim x_{0d} \gg 1$ and $a_{ij} = \delta_{ij}, c_0 =1, F=0, \lambda =0$, the PDE in  \eqref{eq6.12} is approximated by a nonhomogeneous heat equation
\[
u_t  -x_{0d}^{\alpha} \Delta u = f,
\]
which can be reduced to the heat equation with unit heat constant under the scaling
\[
(t,x) \mapsto (s^{2-\alpha} t, s^{1-\alpha/2} x_{0d}^{-\alpha/2}x), \quad s>0.
\]
Due to these facts, throughout the paper, the following notation on parabolic cylinders in $\Omega_T$ are used.
For each $z_0 = (t_0, x_0) \in (-\infty, T) \times \bR^d_+$ with $x_0= (x_0', x_{0d}) \in \bR^{d-1} \times \bR_+$ and $\rho>0$, we write
\begin{equation} \label{Q.def}
\begin{split}
& Q_{\rho}(z_0) =  (t_0 - \rho^{2-\alpha}, t_0) \times B_{r(\rho, x_{0d})} (x_0), \quad \\
&Q_{\rho}^+(z_0) = Q_\rho(z_0) \cap \{x_d>0\},
\end{split}
\end{equation}
where
\begin{equation} \label{r.def}
r(\rho,x_{0d}) = \max\{\rho, x_{0d}\}^{\alpha/2} \rho^{1-\alpha/2}.
\end{equation}
Note that $Q_{\rho}(z_0) = Q_{\rho}^+(z_0) \subset (-\infty, T) \times \bR^d_+$ for $\rho \in (0,x_{0d})$.

 We impose the following assumption  on the partial mean oscillations of the coefficients $(a_{ij})$ and $c_0$, which is an adaptation of the same concept introduced in \cite{MR2338417, MR2300337}.

\begin{assumption}[$\rho_0, \delta$] \label{mean-osc}
For every $\rho \in (0, \rho_0)$ and for each $z= (z', x_{d}) \in \overline{\Omega}_T$, there exist $[a_{ij}]_{\rho, z'},  [c_{0 }]_{\rho, z'}: ((x_{d} -r(\rho, x_d))_+, x_d + r(\rho, x_d)) \rightarrow \bR$ such that \eqref{ellipticity-cond} and \eqref{c-mu.cond} hold on $((x_{d} -r(\rho, x_d))_+, x_d + r(\rho, x_d))$  with  $[a_{ij}]_{\rho, z'}$ in place of $(a_{ij})$ and $[c_{0 }]_{\rho, z'}$ in place of $c_0$.
Moreover,
\[
\begin{split}
& \max_{i, j =1,2,\ldots, d}\fint_{Q_\rho^+(z)} | a_{ij}(\tau, y', y_d) -[a_{ij}]_{\rho,z'}(y_d)|\, dy' dy_d d\tau \\
& \qquad + \fint_{Q_\rho^+(z)} | c_{0}(\tau, y', y_d) -[c_{0}]_{\rho,z'}(y_d)|\, dy'dy_d d\tau < \delta.
\end{split}
\]
\end{assumption}

\subsection{Main results} We now state the main results of the paper.

\begin{theorem} \label{main-thrm}
For given $\nu \in (0,1), \alpha \in (0,2)$ and $p \in (1, \infty)$, there are a sufficiently large number $\lambda_0 = \lambda_0(d, \nu, \alpha, p)>0$ and a sufficiently small number $\delta = \delta(d, \nu, \alpha, p) >0$ such that the following assertions hold.
Assume   \eqref{ellipticity-cond}, \eqref{c-mu.cond}, and \textup{Assumption} \ref{mean-osc} \textup{($\rho_0, \delta$)} are satisfied
with some $\rho_0>0$.
Then for any $F \in L_p(\Omega_T)^d$,  $\lambda \geq \lambda_0 \rho_0^{\alpha-2}$, and $f = f_1 + f_2$ such that $x_d^{1-\alpha} f_1$ and $x_d^{-\alpha/2} f_2\in L_p(\Omega_T)$, there exists a unique weak solution $u \in \sH^{1}_{p}(\Omega_T)$ of \eqref{eq6.12}. Moreover,
\begin{equation} \label{main-est-0508}
\begin{split}
& \|Du\|_{L_p(\Omega_T)} +   \sqrt{\lambda} \| x_d^{-\alpha /2} u\|_{L_p(\Omega_T)}   \leq  N \big(\|F\|_{L_p(\Omega_T)}  +   \|g\|_{L_p(\Omega_T)} \big),
 \end{split}
\end{equation}
where $N = N(\nu, d, \alpha, p)>0$ and $g(z) = x_d^{1-\alpha}|f_1(z)|+ \lambda^{-1/2} x_d^{-\alpha/2} |f_2(z)|$ for $z = (z', x_d) \in \Omega_T$.
\end{theorem}

Our second result is about the estimate and solvability in weighted Sobolev spaces.
For  $p \in (1, \infty)$, we write $w \in A_p (\bR^{d+1}_+)$ if $w$ is a weight on $\bR^{d+1}_+$
such that
$$
    [w]_{A_p (\bR^{d+1}_+)} : = \sup_{ z_0 \in \overline{\bR^{d+1}_+}, \rho > 0}
    \Big(\fint_{    Q^+_\rho (z_0) } w (z) \, dz\Big)\Big(\fint_{ Q^+_\rho (z_0) } w^{-1/(p-1)}(z) \, dz\Big)^{p-1} < \infty.
$$
\begin{theorem}
                \label{thrm-2}
Let $\nu \in (0,1), \alpha \in (0,2)$, $p\in (1, \infty)$ be fixed, and $M \geq 1$.
Assume that $w\in A_p(\bR^{d+1}_+)$ with $[w]_{A_p(\bR^{d+1}_+)}\le M$.
There are a sufficiently large number $\lambda_0 = \lambda_0(d, \nu, \alpha, p,M)>0$ and a sufficiently small number $\delta = \delta(d, \nu, \alpha, p,M) >0$ such that the following assertions hold.
Assume  \eqref{ellipticity-cond}, \eqref{c-mu.cond}, and \textup{Assumption} \ref{mean-osc} \textup{($\rho_0, \delta$)} are satisfied with some $\rho_0>0$.
Then for any $F \in L_p(\Omega_T, w)^d$,  $\lambda \geq \lambda_0 \rho_0^{\alpha-2}$, and $f = f_1 + f_2$ such that $x_d^{1-\alpha} f_1$ and $x_d^{-\alpha/2} f_2\in L_p(\Omega_T,  w)$, there exists a unique weak solution $u \in \sH^{1}_{p}(\Omega_T,w)$ of \eqref{eq6.12}.
Moreover,
\begin{equation} \label{main-est-0508b}
\begin{split}
& \|Du\|_{L_p(\Omega_T,w)} +   \sqrt{\lambda} \| x_d^{-\alpha /2} u\|_{L_p(\Omega_T,w)}   \leq  N \big(\|F\|_{L_p(\Omega_T,w)}  +   \|g\|_{L_p(\Omega_T,w)} \big),
 \end{split}
\end{equation}
where $N = N(\nu, d, \alpha, p,M)>0$ and $g(z) = x_d^{1-\alpha}|f_1(z)|+ \lambda^{-1/2} x_d^{-\alpha/2} |f_2(z)|$ for $z = (z', x_d) \in \Omega_T$.
\end{theorem}

\begin{remark}
We can similarly define $A_p(\bR^{d}_+)$ and $A_p(\bR^d)$ with half balls $B_\rho^+(x_0)$ and balls $B_\rho(x_0)$ in place of $Q^+_\rho (z_0)$,  respectively.
It is easily seen that if $w_1\in A_p (\bR)$ and $w_2\in A_p(\bR^d_+)$, then $w=w(t,x):=w_1(t)w_2(x)\in A_p(\bR^{d+1}_+)$ and
$$
[w]_{A_p (\bR^{d+1}_+)}\le [w_1]_{A_p (\bR)}[w_2]_{A_p (\bR^{d}_+)}.
$$
Consequently, by using the Rubio de Francia extrapolation theorem (see, for instance, \cite{MR745140} or \cite[Theorem 2.5]{MR3812104}), from Theorem \ref{thrm-2}, we also derive the corresponding weighted mixed-norm estimate and solvability.
We also mention that a typical example of such $A_p$ weight $w_2$ is given by $x_d^\gamma$ for any $\gamma\in (-1,p-1)$.
\end{remark}

\begin{remark}
Theorems \ref{main-thrm} and \ref{thrm-2} can be extended to equations with lower-order terms in the form
$$
u_t+\lambda c_0(z) u - \mu(x_d)D_i\big( a_{ij} (z) D_j u)+b_iD_i u+cu=f+\mu(x_d)D_iF_i \quad \text{in} \quad \Omega_T,
$$
where $b$ and $c$ are bounded and measurable, and $b\equiv 0$ when $\alpha\in [1,2)$.
To see this, we write the equation into
$$
u_t+ \lambda c_0(z) u - \mu(x_d) D_i\big( a_{ij} (z) D_j u)=\tilde f+\mu(x_d) D_iF_i \quad \text{in} \quad \Omega_T,
$$
where
$$
\tilde f=\tilde f_1+\tilde f_2,\quad \tilde f_1=f_1-b_iD_i u 1_{x_d<\tau},\quad\tilde f_2=f_2-b_iD_i u 1_{x_d\ge \tau}-cu.
$$
By the theorems above, we have
\begin{align*}
&\|Du\| +   \sqrt{\lambda} \| x_d^{-\alpha /2} u\|  \\
&\leq  N \big(\|F\|  +   \|g\|+\|x_d^{1-\alpha} b_iD_i u1_{x_d<\tau}\|
+\lambda^{-1/2}
\|x_d^{-\alpha/2} (b_iD_i u1_{x_d\ge\tau}+cu)\|\big)\\
&\leq  N \big(\|F\|  +   \|g\|\big)+N(\tau^{1-\alpha}+\lambda^{-1/2}\tau^{-\alpha/2})\|bD u\|+N\lambda^{-1/2}\|x_d^{-\alpha/2}u\|,
\end{align*}
where $\|\,\cdot\,\|$ is either the $L_p$ norm or the weighted $L_p$ norm and $N$ is independent of $\tau$.
By taking $\tau$ sufficiently small and then $\lambda$ sufficiently large, we can absorb the second and last terms on the right-hand side to the left-hand side.
The solvability then follows from the method of continuity.
Finally, we can also deduce the corresponding results for elliptic equations of the form
$$
- D_i\big( a_{ij} (z) D_j u)+ \mu(x_d)^{-1}(b_iD_iu +cu+\lambda c_0(z) u)=\mu(x_d)^{-1} f+D_iF_i \quad \text{in} \quad \bR^d_+
$$
with the Dirichlet boundary condition $u=0$ on $\{x_d=0\}$,
by viewing solutions to the elliptic equations as steady state solutions to the corresponding parabolic equations.
We refer the reader to the proof of \cite[Theorem 2.6]{Krylov-07}.
It is worth noting that here the lower-order coefficients $\mu(x_d)^{-1}b$ and $\mu(x_d)^{-1} c$ do not even belong to $L_d$ and $L_{d/2}$, respectively, when $\alpha \in [2/d, 2)$, which are usually required in the classical $L_p$ theory.
See, for instance, \cite{kr21} and the references therein.
\end{remark}
\begin{remark} We note that $W^1_p(\bR^d_+) = \sW^1_p(\bR^d_+)$ if $p \geq 2/\alpha$. Moreover, the estimate \eqref{main-est-0508} also implies that
\[
\|x_d^{-1}u\|_{L_p(\Omega_T)} \leq N \big(\|F\|_{L_p(\Omega_T)}  +   \|g\|_{L_p(\Omega_T)} \big)
\]
due to Hardy's inequality.
\end{remark}


\section{Preliminary Analysis and \texorpdfstring{$L_2$}{}-solutions} \label{sec-3}

\subsection{A filtration of partitions and a quasi-metric}  \label{filtration.def}
We construct a filtration of partitions $\{\bC_n\}_{n \in \bZ}$ (dyadic decompositions) of $\bR\times \bR^d_+$, which satisfies the following three basic conditions (see \cite{Krylov}):
\begin{enumerate}[(i)]
\item The elements of partitions are ``large'' for big negative $n$'s and ``small''
for big positive $n$'s: for any $f\in L_{1,\text{loc}}$,
$$
\inf_{C\in \bC_n}|C|\to \infty\quad\text{as}\,\,n\to -\infty,\quad
\lim_{n\to \infty}(f)_{C_n(z)}=f(z)\quad\text{(a.e.)},
$$
where $C_n(z)\in \bC_n$ is such that $z\in C_n(z)$.

\item The partitions are nested: for each $n\in \bZ$, and $C \in \bC_n$, there exists a unique $C' \in \bC_{n-1}$ such that $C \subset C'$.

\item Moreover, the following regularity property holds: For $n,C, C'$ as in (ii), we have
$$
|C'|\le N_0|C|,
$$
where $N_0$ is independent of $n$, $C$, and $C'$.
\end{enumerate}

For any $s\in \bR$, denote $\lfloor s \rfloor$ to be the integer part of $s$, i.e., the largest integer which is less than or equal to $s$. For a fixed $\alpha\in (0,2)$ and $n\in \bZ$, let $k_0=\lfloor -n/(2-\alpha) \rfloor$. We construct $\bC_n$ as follows: it contains boundary cubes in the form
$$
((j-1)2^{-n},j2^{-n}]\times (i_12^{k_0},(i_1+1)2^{k_0}]
\times\cdots\times (i_{d-1}2^{k_0},(i_{d-1}+1)2^{k_0}]\times (0, 2^{k_0}],
$$
where $j,i_1,\ldots,i_{d-1}\in \bZ$, and interior cubes in the form
$$
((j-1)2^{-n},j2^{-n}]\times (i_12^{k_2},(i_1+1)2^{k_2}]
\times\cdots \times (i_d2^{k_2}, (i_d+1)2^{k_2}],
$$
where $j,i_1,\ldots,i_{d}\in \bZ$ and
\begin{equation}
                    \label{eq5.07}
i_d2^{k_2}\in [2^{k_1},2^{k_1+1})\, \text{for some integer}\, k_1\ge k_0,
\quad k_2=\lfloor (-n+k_1\alpha)/2 \rfloor-1.
\end{equation}
Note that $k_2$ is increasing with respect to $k_1$ and decreasing with respect to $n$. Because $k_1\ge k_0>-n/(2-\alpha)-1$, we have
$(-n+k_1\alpha)/2-1\le k_1$,
which implies that $k_2\le k_1$ and $(i_d+1)2^{k_2}\le 2^{k_1+1}$. It is easily seen that all three conditions above are satisfied. Furthermore, according to \eqref{eq5.07} we also have
$$
(2^{k_2}/2^{k_1})^2\sim 2^{-n}/(2^{k_1})^{2-\alpha},
$$
which allows us to apply the interior estimates after a scaling.

Next we define a function $\varrho: \Omega_\infty\times \Omega_\infty\to  [0,\infty)$:
$$
 \varrho((t,x),(s,y))=|t-s|^{1/(2-\alpha)}
+\min\big\{|x-y|,|x-y|^{2/(2-\alpha)}\min\{x_d,y_d\}^{-\alpha/(2-\alpha)}\big\}.
$$
It is easily seen that $\varrho$ is a quasi-metric on $\Omega_\infty$, i.e., there exists a constant $K_1=K_1(d,\alpha)>0$ such that
$$
 \varrho((t,x),(s,y))\le K_1\big(\varrho((t,x),(\hat t,\hat x))+ \varrho((\hat t,\hat x),(s,y))\big)
$$
for any $(t,x),(s,y),(\hat t,\hat x)\in \Omega_\infty$,  and $ \varrho((t,x),(s,y))=0$ if and only if $(t,x)=(s,y)$.
Moreover, the cylinder $Q_\rho^+(z_0)$ defined in \eqref{Q.def} is comparable to
$$
\{(t,x)\in \Omega: t<t_0,\, \varrho((t,x),(t_0,x_0))<\rho \}.
$$
Therefore, $(\Omega_T, \varrho)$ equipped with the Lebesgue measure is a space of homogeneous type and we have a dyadic decomposition, which is given above.

For a locally integrable function $f$ defined on a domain $Q\subset \bR^{d+1}$, we write
\[
(f)_{Q} = \fint_{Q} h(s,y)\, dyds.
\]
We define the dyadic maximal function and sharp function of a locally integrable function $f$ in $\Omega_\infty$ by
\begin{align*}
\cM_{\text{dy}} f(z)&=\sup_{n<\infty}\fint_{C_n(z)\in \bC_n}|f(s,y)|\,dyds,\\
f_{\text{dy}}^{\#}(z)&=\sup_{n<\infty}\fint_{C_n(z)\in \bC_n}|f(s,y)-(f)_{C_n(z)}|\,dyds.
\end{align*}
We also define the maximal function and sharp function over cylinders by
\begin{align*}
\cM f(z)&=\sup_{z\in Q^+_\rho(z_0), z_0\in \overline{\Omega_\infty}}\fint_{Q_\rho^+(z_0)}|f(s,y)|\,dyds,\\
f^{\#}(z)&=\sup_{z\in Q^+_\rho(z_0),z_0\in \overline{\Omega_\infty}}\fint_{Q_\rho^+(z_0)}|f(s,y)-(f)_{Q^+_\rho(z_0)}|\,dyds.
\end{align*}
It is easily seen that for any $z\in \Omega_\infty$, we have
$$
\cM_{\text{dy}} f(z)\le N\cM f(z),\quad f_{\text{dy}}^{\#}(z)\le Nf^{\#}(z),
$$
where $N=N(d,\alpha)$.

\subsection{\texorpdfstring{$L_2$}{}-solutions}  We begin with the following lemma on the energy estimate for \eqref{eq6.12}.

\begin{lemma} \label{lemma-ener-1}
Suppose that \eqref{ellipticity-cond} and \eqref{c-mu.cond} are satisfied,  $F \in L_2(\Omega_T)^d$, and $\lambda>0$.
Also let $f = f_1 + f_2$ such that $x_d^{1-\alpha} f_1$ and $x_d^{-\alpha/2} f_2$ are in $L_2(\Omega_T)$.
If $u \in  \sH^{1}_2(\Omega_T)$ is a weak solution of  \eqref{eq6.12}, then
\begin{equation} \label{1016-2c.est}
\begin{split}
& \|Du\|_{L_2(\Omega_T)} + \sqrt{\lambda} \| x_d^{-\alpha/2}u\|_{L_2(\Omega_T)} \leq N \Big[ \|F\|_{L_2(\Omega_T)} + \|g\|_{L_2(\Omega_T)}  \Big],
\end{split}
\end{equation}
where $N = N(\nu, d)$ and $g(z)= x_d^{1-\alpha} |f_1(z)| + \lambda^{-1/2}x_d^{-\alpha/2} |f_2(z)|$ for $z = (z', x_d) \in \Omega_T$.
\end{lemma}

\begin{proof}
By using the Steklov averages, we can formally take $u$ as the test function in Definition \ref{weak-sol-def}.
Then, it follows from \eqref{ellipticity-cond} and \eqref{c-mu.cond} that
\begin{align} \notag
&   \frac{d}{dt} \int_{\bR^{d}_+} \mu(x_d)^{-1} |u|^2  \,dx + \lambda \int_{\bR^{d}_+} x_d^{-\alpha} |u|^2 \,dx +   \int_{\bR^{d}_+} |Du|^2 \,dx \\ \label{u-test.0628}
& \leq N(\nu, d) \int_{\bR^{d}_+}  \big( |u| |f| x_d^{-\alpha}  + |F| |Du|\big)\, dx.
\end{align}
Now, we control the right-hand side of \eqref{u-test.0628}.
By using Young's inequality and Hardy's inequality for terms on the right-hand side, we see that
\[
\begin{split}
& N(\nu, d) \int_{\bR^{d}_+}  \big( |u| |f| x_d^{-\alpha}  + |F| |Du|\big)\, dx \\
& \le  N(\nu, d) \int_{\bR^{d}_+}  \big( |u/x_d| |f_1| x_d^{1-\alpha}  +  |\lambda^{1/2}x_d^{-\alpha/2}u| |\lambda^{-1/2} x_d^{-\alpha/2}f_2| + |F| |Du|\big)\, dx  \\
& \leq \frac{1}{2} \int_{\bR^d_+} \big(|Du|^2 + \lambda x_d^{-\alpha} u^2\big)\, dx \\
& \qquad + N(d, \nu) \int_{\bR^d_+} \Big[|x_d^{1-\alpha} f_1|^2 +  \lambda^{-1} x_d^{-\alpha} |f_2|^2 + |F|^2 \Big]\, dx.
\end{split}
\]
It then follows from \eqref{u-test.0628} that
\[
\begin{split}
&   \frac{d}{dt} \int_{\bR^{d}_+} \mu(x_d)^{-1} |u|^2  \,dx + \lambda \int_{\bR^{d}_+} x_d^{-\alpha} |u|^2 \,dx +   \int_{\bR^{d}_+} |Du|^2 \,dx \\
& \leq  N(\nu, d) \int_{\bR^d_+} \big(|x^{1-\alpha} f_1|^2 + \lambda^{-1} |x^{-\alpha/2} f_2|^2 + |F|^{2} \big)\, dx.
\end{split}
\]
Now, by integrating the the above inequality with respect to the time variable, we obtain \eqref{1016-2c.est}.
The lemma is proved.
\end{proof}
We prove the following simple but important result in this subsection.
\begin{theorem} \label{L-2.theorem} Let $\nu \in (0,1), \alpha \in (0,2)$, $\lambda >0$, and $F \in L_2(\Omega_T)^d$.
Also let $f = f_1 + f_2$ and assume that $x_d^{1-\alpha} f_1$ and $x_d^{-\alpha/2} f_2$ are in $L_2(\Omega_T)$.
If \eqref{ellipticity-cond} and \eqref{c-mu.cond} are satisfied, then there exists a unique weak solution $u \in \sH^{1}_2 (\Omega_T)$ of  \eqref{eq6.12}.
Moreover,
\begin{equation} \label{1016-2.est}
\begin{split}
& \|Du\|_{L_2(\Omega_T)} + \sqrt{\lambda} \| x_d^{-\alpha/2} u\|_{L_2(\Omega_T)}   \leq N \Big[ \|F\|_{L_2(\Omega_T)} + \|g\|_{L_2(\Omega_T)}   \Big],
\end{split}
\end{equation}
where $N = N(\nu, d)$ and  $g(z)= x_d^{1-\alpha} |f_1(z)| + \lambda^{-1/2}x_d^{-\alpha/2} |f_2(z)|$ for $z = (z', x_d) \in \Omega_T$.
\end{theorem}
\begin{proof}
We approximate the domain $\Omega_T$ by a sequence of increasing bounded domains $\{\widehat{Q}_k\}_k$ given by
\begin{equation*} 
\widehat{Q}_k= (-k, \min\{k, T\}) \times B_k^+, \quad k \in \mathbb{N}.
\end{equation*}
For each fixed  $k \in \mathbb{N}$, we consider the equation of $u$ in $\widehat{Q}_k$
\begin{equation} \label{w-eqn.Qn}
 u_t + \lambda c_0(z)u  - \mu(x_d) D_i\big( a_{ij}(z) D_{j} u + F_i\big) = f(z) \quad \text{in} \quad \widehat{Q}_k
\end{equation}
with the boundary condition $u=0$ on $ (-k, \min\{k, T\}) \times \partial B_k^+$ and zero initial data at $\{-k\} \times B_k^+$.
Then, using the energy estimates as in the proof of Lemma \ref{lemma-ener-1},  if $u_k \in \sH^{1}_2(\widehat{Q}_k)$ is a weak solution of \eqref{w-eqn.Qn}, we have the following a priori estimate
\[
\begin{split}
& \|x_d^{-\alpha} u_k\|_{L_\infty((-k, \min\{k, T\}), L_2(B_k^+))} + \sqrt{\lambda} \|x_d^{-\alpha/2} u_k\|_{L_2(\widehat{Q}_k)} + \|Du_k\|_{L_2(\widehat{Q}_k)} \\
&\leq N \Big[  \|F\|_{L_2(\widehat{Q}_k)} + \|x_d^{1-\alpha}f_1\|_{L_2(\widehat{Q}_k)}  + \lambda^{-1/2} \|x_d^{-\alpha/2}f_2\|_{L_2(\widehat{Q}_k)}\Big]
 \end{split}
\]
for $N = N(d, \nu)>0$.
From this and  the Galerkin method, we see that for each $k \in \bN$, there exists a unique weak solution $u_k \in \sH^{1}_2(\widehat{Q}_k)$ of \eqref{w-eqn.Qn}.
By taking $u_k =0$ in $\Omega_T \setminus \widehat{Q}_k$, we see that $u_k$ as a function defined in $\Omega_T$ satisfying
\[
\begin{split}
& \|{x_d^{-\alpha}}u_k\|_{L_\infty((-\infty, T), L_2(\bR^d_+))} + \sqrt{\lambda} \|x_d^{-\alpha/2} u_k\|_{L_2(\Omega_T)} + \|Du_k\|_{L_2(\Omega_T)} \\
&\leq N \Big[ \|F\|_{L_2(\Omega_T)} + \|x_d^{1-\alpha} f_1\|_{L_2(\Omega_T)} +  \lambda^{-1/2} \|x_d^{-\alpha/2} f_2\|_{L_2(\Omega_T)} \Big].
 \end{split}
\]
From this, and by taking a subsequence still denoted by $\{u_k\}$, we can find $u \in \sH^1_2(\Omega_T)$ such that
\[
\begin{split}
& u_k \rightharpoonup u  \quad \text{ in } \quad L_2(\Omega_T, x_d^{-\alpha})  \quad \text{as} \quad k \rightarrow \infty,\\
& Du_k \rightharpoonup Du  \quad \text{ in } \quad L_2(\Omega_T)   \quad \text{as} \quad k \rightarrow \infty.
\end{split}
\]
Then, using the weak formulation in Definition \ref{weak-sol-def} and passing to the limit, we see that $u \in \sH^1_2(\Omega_T)$ is a weak solution of \eqref{eq6.12} and satisfies \eqref{1016-2.est}.
Note that the uniqueness of $u \in \sH^{1}_2(\Omega_T)$ also follows from this estimate, and  therefore the proof of the theorem is completed.
\end{proof}


\section{Equations with coefficients depending only on the \texorpdfstring{$x_d$}{}-variable} \label{sec-4}

Let  $\overline{c}_0: \bR_+ \rightarrow \bR_+$ be measurable satisfying
\begin{equation} \label{a-c.cond}
\nu \leq  \overline{c}_0(x_d) \leq \nu^{-1} \quad \text{ for } x_d \in \bR_+
\end{equation}
for a given constant $\nu \in (0,1)$.
Also, let $(\overline{a}_{ij})_{i,j=1}^d: \bR_+ \rightarrow \mathbb{R}^{d \times d}$ be a  matrix of measurable functions satisfying the following ellipticity  and boundedness conditions
\begin{equation} \label{elli-cond}
\nu|\xi|^2 \leq \overline{a}_{ij}(x_d) \xi_i \xi_j, \quad |\overline{a}_{ij}(x_d)| \leq \nu^{-1}  \quad \text{ for } x_d \in \bR_+ \end{equation}
and $\xi =(\xi_1, \xi_2, \ldots, \xi_d) \in \bR^d$.
For a fixed number $\lambda > 0$, let us denote
\[
\sL_0 u =u_t + \lambda \overline{c}_0(x_d) u - \mu(x_d) D_i\big( \overline{a}_{ij}(x_d) D_{j} u),
\]
where $\mu$ satisfies \eqref{c-mu.cond}.
We study the following equation
\begin{equation} \label{x-d.eqn}
\left\{
\begin{array}{rccl}
 \sL_0 u & = & \mu(x_d) D_i F_i  + f   & \quad \text{in} \quad \Omega_T,  \\
u  & = & 0 & \quad \text{on} \quad \{x_d =0\},
\end{array} \right.
\end{equation}
which is a simple form of \eqref{eq6.12} as  the coefficients only depend on $x_d$.

\medskip

The main result of this section is the following theorem, which is a special case of Theorem \ref{main-thrm}.

\begin{theorem}\label{sim-eqn-thrm}
Let $\nu \in (0,1)$, $\alpha \in (0,2)$, $\lambda >0$, and suppose that \eqref{c-mu.cond}, \eqref{a-c.cond}, and \eqref{elli-cond} are satisfied.
Also, let $F = (F_1, F_2, \ldots, F_d) \in L_p(\Omega_T)^d$,  $f = f_1 + f_2$ such that $x_d^{1-\alpha} f_1$ and $x_d^{-\alpha/2} f_2$ are in $L_p(\Omega_T)$, where $p \in (1, \infty)$.
Then, there exists a unique  weak solution $u \in \sH^{1}_p (\Omega_T)$ of  \eqref{x-d.eqn}.
Moreover, there is a constant $N = N(\nu, d, \alpha, p) >0$ such that
\begin{equation} \label{apr-est}
\begin{split}
& \|Du\|_{L_p(\Omega_T)} + \sqrt{\lambda} \|x_d^{-\alpha/2} u\|_{L_p(\Omega_T)}   \\
&  \leq N\Big[ \|F\|_{L_p(\Omega_T)} +  \|x_d^{1-\alpha} f_1\|_{L_p(\Omega_T)}  + \lambda^{-1/2}\|x_d^{-\alpha/2} f_2\|_{L_p(\Omega_T)}\Big].
\end{split}
\end{equation}
\end{theorem}

The rest of the section is to prove Theorem \ref{sim-eqn-thrm}.
Our idea is to first establish mean oscillation estimates and then use the Fefferman-Stein theorem on sharp functions and the Hardy-Littlewood maximal function theorem in spaces of homogeneous type.
It is therefore important to derive regularity estimates for homogeneous equations.
In the next two subsections (Subsections \ref{bdr-schauder} and  \ref{int-schauder}), we derive the boundary H\"older estimates and interior H\"older estimates for solutions to homogeneous equations.
The mean oscillation estimates of solutions and the proof of Theorem  \ref{sim-eqn-thrm} are given in Subsection \ref{oss-est-simple-coffe}.

\subsection{Boundary H\"older estimates for homogeneous equations} \label{bdr-schauder}
In this subsection, we consider the following homogeneous equation
\begin{equation}
                    \label{eq6.12h}
                    \left\{
                    \begin{array}{rccl}
\sL_0 u & = & 0 & \quad \text{in} \quad Q_1^+, \\
u & = & 0 & \quad \text{on} \quad Q_1 \cap \{x_d =0\}.
\end{array} \right.
\end{equation}
Our goal is to prove Proposition \ref{prop3} below on H\"older estimates for weak solutions.
We begin with the following local energy estimate.
\begin{lemma}[Energy inequality]
                        \label{lem1}
Suppose that \eqref{c-mu.cond},  \eqref{a-c.cond}, and \eqref{elli-cond} are satisfied in $Q_1^+$.
If $u\in \sH^1_2(Q_1^+)$ is a weak solution of \eqref{eq6.12h} in $Q_1^+$, then
\begin{equation}
                    \label{eq6.31}
\sup_{s\in (-1/2,0)}\int_{B_{1/2}^+}u^2(s,x) x_d^{-\alpha}\,dx+
\int_{Q_{1/2}^+}(\lambda u^2 x_d^{-\alpha}+|Du|^2)\,dz
\le N \int_{Q_{1}^+} u^2 \,dz,
\end{equation}
where $N=N(d,\nu,\alpha)>0$.
\begin{proof}
Let $\eta\in C_0^\infty((-1,1))$ and $\zeta\in C_0^\infty(B_1)$ be nonnegative functions such that $\eta=1$ on $(-1/2,1/2)$ and $\zeta=1$ on $B_{1/2}$.
We test \eqref{eq6.12h} by $u\mu^{-1} \eta^\beta(t)\zeta^2(x)$, where $\beta=2/(2-\alpha)$, and integrate by parts. We then get
\begin{align}
                        \label{eq7.14}
&\sup_{s\in (-1,0)}\int_{B_{1/2}^+}u^2(s,x) x_d^{-\alpha}\eta^\beta(s)\zeta^2(x)\,dx+
\int_{Q_{1}^+}(\lambda u^2 x_d^{-\alpha}+|Du|^2)\eta^\beta\zeta^2 \,dz\notag\\
&\le N\int_{Q_{1}^+}  u^2x_d^{-\alpha}\eta^{\beta-1}|\eta_t|\zeta^2 +|Du||u|\eta^\beta \zeta |D\zeta|\,dz.
\end{align}
To estimate the first term on the right-hand side, we use H\"older's inequality to get
\begin{align}
                    \label{eq6.43}
&N\int_{Q_{1}^+}  u^2x_d^{-\alpha}\eta^{\beta-1}|\eta_t|\zeta^2\,dz\notag\\
&\le N\Big(\int_{Q_{1}^+} u^2x_d^{-2}\eta^\beta\zeta^2\,dz\Big)^{\alpha/2}\Big(\int_{Q_{1}^+} u^2\zeta^2\,dz\Big)^{1-\alpha/2}\notag\\
&\le N\Big(\int_{Q_{1}^+} (|D_du|^2\zeta^2+u^2|D_d\zeta|^2)\eta^\beta\,dz\Big)^{\alpha/2}\Big(\int_{Q_{1}^+} u^2\zeta^2\,dz\Big)^{1-\alpha/2}\notag\\
&\le \frac 1 3 \int_{Q_{1}^+} |Du|^2\zeta^2\eta^\beta\,dz+ N\int_{Q_{1}^+} u^2\,dz,
\end{align}
where we used $\beta-1=\alpha\beta/2$ in the first inequality, Hardy's inequality in the second inequality, and Young's inequality in the last inequality.
By Young's inequality, the second term on the right-hand side of \eqref{eq7.14} is bounded by
\begin{equation}
                        \label{eq7.16}
N\int_{Q_{1}^+} |Du||u|\eta^\beta \zeta |D\zeta|\,dz
\le \frac 1 3 \int_{Q_{1}^+} |Du|^2 \eta^\beta \zeta^2\,dz+N\int_{Q_{1}^+} u^2\,dz.
\end{equation}
Combining \eqref{eq7.14}, \eqref{eq6.43}, and \eqref{eq7.16}, we get \eqref{eq6.31}.
The lemma is proved.
\end{proof}
\end{lemma}

\begin{lemma}
                \label{lem2}
Under the conditions of Lemma \ref{lem1}, we have
\begin{equation}
                        \label{eq7.28}
\int_{Q_{1/2}^+}u_t^2 x_d^{-\alpha}\,dz
\le N \int_{Q_{1}^+} u^2 \,dz,
\end{equation}
where $N=N(d,\nu,\alpha)>0$.
\end{lemma}
\begin{proof}
We test the equation with $u_t\mu^{-1} \eta^\beta(t)\zeta^2(x)$, integrate by parts, and use Lemma \ref{lem1} by noting that $u_t$ satisfies the same equation as $u$ with the same boundary condition on $\{x_d=0\}$ and a standard iteration argument.
\end{proof}

Recall that for each $\beta \in (0,1)$, the $\beta$-H\"older semi-norm in  the spatial variable of a function $u$ on an open set $Q\subset \bR^{d+1}$ is defined by
\[
\begin{split}
\llbracket u\rrbracket_{C^{0, \beta}(Q)} = \sup\Big\{ & \frac{|u(t,x) - u(t,y)|}{|x-y|^{\beta}}:  x \not =y, \  (t,x), (t,y) \in Q \Big\}.
\end{split}
\]
For $k, l \in \mathbb{N} \cup \{0\}$,  we denote
\[
\|u\|_{C^{k, l}(Q)}  = \sum_{i=0}^k \sum_{|j| \leq l}\|\partial_t^i D_{x}^j u\|_{L_\infty(Q)} .
\]
Moreover, the following notation for  the H\"{o}lder norm of $u$ on $Q$ is used
\[
\|u\|_{C^{k, \beta}(Q)} = \|u\|_{C^{k,0}(Q)} + \sum_{i=0}^{k} \llbracket \partial_t^i u\rrbracket_{C^{0, \beta}(Q)}.
\]

\begin{corollary}
                                    \label{cor1}
Under the conditions of Lemma \ref{lem1}, for any integer $k\ge 0$, we have
\begin{equation}
                        \label{eq7.37}
\|u\|_{C^{k,1/2}(Q_{1/2}^+)}
\le N \|u\|_{L_2(Q_1^+)},\quad
\|D_{x'}u\|_{C^{k,1/2}(Q_{1/2}^+)}\le N \|D_{x'}u\|_{L_2(Q_1^+)},
\end{equation}
where $N=N(d,\nu,\alpha,k)>0$.
\end{corollary}
\begin{proof}
From Lemmas \ref{lem1} and \ref{lem2}, by induction we have
\begin{equation}
                \label{eq8.11}
\int_{Q_{1/2}^+}|\partial_t^k D_{x'}^j D_{d}^l u|^2 \,dz
\le N(d,\nu,\alpha,k,j,l) \int_{Q_{1}^+} u^2 \,dz
\end{equation}
for any integers $k,j\ge 0$ and $l=0,1$.
Then the first inequality in \eqref{eq7.37} follows from the Sobolev embedding theorem.
The second inequality follows from the first one by noting that $D_{x'} u$ satisfies the same equation as $u$ with the same boundary condition on $\{x_d=0\}$.
\end{proof}

Next, we show higher regularity of $u$.
\begin{proposition}  \label{prop3}
Under the conditions of Lemma \ref{lem1}, we have
\begin{equation}
                        \label{eq8.10}
\|u\|_{C^{1,1}(Q_{1/2}^+)}+
\|D_{x'}u\|_{C^{1,1}(Q_{1/2}^+)}+\|U\|_{C^{1,\gamma}(Q_{1/2}^+)}
\le N \|Du\|_{L_2(Q_1^+)}
\end{equation}
and
\begin{equation}
                        \label{eq8.16}
\sqrt\lambda\|u x_d^{-\alpha/2}\|_{C^{1,1-\alpha/2}(Q_{1/2}^+)}\le N \|Du\|_{L_2(Q_1^+)},
\end{equation}
where $N=N(d,\nu,\alpha)>0$, $\gamma=\min\{2-\alpha,1\}$, and $U(z)=\overline{a}_{dj}(x_d)D_j u(z)$ for $z = (z', x_d) \in Q_1^+$.
\end{proposition}
\begin{proof}
Let $\beta=2(\alpha-1)_+\in [0,2)$. Using \eqref{eq8.11}, we have
\begin{equation}
                \label{eq8.11b}
\int_{Q_{1/2}^+}|\partial_t^k D_{x'}^j U|^2 \,dz
\le N(d,\nu,\alpha,k,j) \int_{Q_{1}^+} u^2 \,dz
\end{equation}
for any integers $k,j\ge 0$.
From the equation \eqref{eq6.12h},
\begin{equation}
                            \label{eq8.50}
D_d U=\mu(x_d)^{-1}(u_t+\lambda \overline{c}_0 u)-\sum_{i=1}^{d-1}\sum_{j=1}^{d}\overline{a}_{ij}(x_d)D_{ij}u.
\end{equation}
Therefore, for $r\in (1/2,1)$,
\begin{align*}
&\int_{Q_r^+}|D_d U|^2 x_d^\beta\,dz
\le N \int_{Q_r^+} (|u_t|+\lambda |u|)^2 x_d^{-2\alpha+\beta}
+|DD_{x'}u|^2 x_d^{\beta}\,dz\\
&\le N \int_{Q_r^+} (|u_t|+\lambda |u|)^2 x_d^{-2}
+|DD_{x'}u|^2 \,dz\\
&\le N \int_{Q_r^+} |D_d u_t|^2 +\lambda^2 |D_du|^2+|DD_{x'}u|^2 \,dz
\le N \int_{Q_1^+} |u|^2 \,dz,
\end{align*}
where we used Hardy's inequality in the third inequality, and \eqref{eq6.31} and \eqref{eq8.11} in the last inequality.
Since $\partial_t^k D_{x'}^j u$ satisfies the same equation with the same boundary condition, similarly we have
\begin{equation}
                    \label{eq9.00}
\int_{Q_r^+}|\partial_t^k D_{x'}^j D_d U|^2 x_d^\beta\,dz\le N(d,\nu,\alpha,k,j,r)\int_{Q_1^+} |u|^2 \,dz
\end{equation}
for any integers $k,j\ge 0$ and $r\in (1/2,1)$. Now if $\alpha<3/2$ so that $\beta<1$, by  \eqref{eq9.00} and H\"older's inequality,
$$
\int_{Q_r^+}|\partial_t^k D_{x'}^j D_d U| \,dz\le N(d,\nu,\alpha,k,j,r)\Big(\int_{Q_1^+} |u|^2 \,dz\Big)^{1/2},
$$
which, together with \eqref{eq8.11b} and the Sobolev embedding theorem, implies that
\begin{equation}
                            \label{eq12.20}
\|U\|_{L_\infty(Q_r^+)}\le N\|u\|_{L_2(Q_1^+)}.
\end{equation}
Using the definition of $U$, \eqref{eq7.37}, \eqref{eq12.20}, and the Poincar\'e inequality, we get
\begin{equation}
                        \label{eq8.10a}
\|u\|_{C^{1,1}(Q_{1/2}^+)}
\le N \|u\|_{L_2(Q_1^+)}\le N \|D_d u\|_{L_2(Q_1^+)}.
\end{equation}

If $\alpha\in [3/2,2)$, we employ a bootstrap argument.
By the (weighted) Sobolev embedding (see, for instance, \cite[Theorem 6]{Ha} or \cite[Lemma 3.1]{DP20}) in the $x_d$-variable and the standard Sobolev embedding in the other variables, we get from \eqref{eq9.00} that for any $p_1\in (2,\infty)$ satisfying $1/p_1>1/2-1/(1+\beta)$,
\begin{equation}
                    \label{eq10.25}
\|U\|_{L_{p_1}(Q_r^+,x_d^\beta\,dz)}\le N\|u\|_{L_2(Q_1^+)}.
\end{equation}
Using the definition of $U$, \eqref{eq7.37}, and \eqref{eq10.25}, we get
\begin{equation}
                    \label{eq10.28}
\|Du\|_{L_{p_1}(Q_r^+,x_d^\beta\,dz)}\le N\|u\|_{L_2(Q_1^+)}.
\end{equation}
As before, since $\partial_t^k D_{x'}^j u$ satisfies the same equation, from \eqref{eq10.28} and \eqref{eq8.11}, we obtain
\begin{equation}
                    \label{eq10.28b}
\|\partial_t^k D_{x'}^j Du\|_{L_{p_1}(Q_r^+,x_d^\beta\,dz)}\le N\|u\|_{L_2(Q_1^+)}.
\end{equation}
Since $\beta<2$, we may take $p_1\ge 6$.
Let $\beta_1:=\beta+(\alpha-1)p_1=(\alpha-1)(2+p_1)>\beta$.
 Using \eqref{eq8.50} again, we have
\begin{align}
                \label{eq5.26}
&\int_{Q_r^+}|D_d U|^{p_1} x_d^{\beta_1}\,dz
\le N \int_{Q_r^+} (|u_t|+\lambda |u|)^{p_1} x_d^{-p_1\alpha+\beta_1}
+|DD_{x'}u|^{p_1} x_d^{\beta_1}\,dz\notag\\
&\le N \int_{Q_r^+} ((|u_t|+\lambda |u|)/x_d)^{p_1} x_d^{\beta}
+|DD_{x'}u|^{p_1} x_d^{\beta}\,dz\notag\\
&\le N \int_{Q_r^+} (|D_d u_t| +\lambda |D_du|)^{p_1}x_d^{\beta}+|DD_{x'}u|^{p_1}x_d^\beta \,dz,
\end{align}
where we used the weighted Hardy inequality, which holds true because
$$
(\beta+1)/p_1<3/6<1.
$$
Since $u_t$ and $D_{x'}u$ satisfy the same equation as $u$, by \eqref{eq5.26}, \eqref{eq10.28b}, \eqref{eq6.31}, and \eqref{eq7.28}, we further obtain
$$
\int_{Q_r^+}|D_d U|^{p_1} x_d^{\beta_1}\,dz\le N\Big(\int_{Q_1^+} |u|^2 \,dz\Big)^{p_1/2}.
$$
Similar to \eqref{eq9.00}, from the above inequality we deduce
\begin{equation}
                                \label{eq12.59}
\int_{Q_r^+}|\partial_t^k D_{x'}^j D_d U|^{p_1} x_d^{\beta_1}\,dz\le N\Big(\int_{Q_1^+} |u|^2 \,dz\Big)^{p_1/2}
\end{equation}
for any integers $k,j\ge 0$ and $r\in (1/2,1)$.
Now if $p_1>\beta_1+1$, as before we conclude \eqref{eq12.20} and thus \eqref{eq8.10a} by using \eqref{eq12.59} and  H\"older's inequality. Otherwise, we find $p_2\in (p_1,\infty)$ such that $1/p_2=1/p_1-1/(1+\beta_1)+\varepsilon_1$, where $\varepsilon_1>0$ is a sufficiently small number to be chosen later, and let $\beta_2=\beta_1+(\alpha-1)p_2$.
We repeat this procedure and define $p_k$ and $\beta_k$ recursively for $k\ge 3$ by
$$
1/p_k=1/p_{k-1}-1/(1+\beta_{k-1})+\varepsilon_{k-1},\quad \beta_k=\beta_{k-1}+(\alpha-1)p_k,
$$
where $\varepsilon_k>0$ is a sufficiently small number to be chosen later, until $p_k>\beta_k+1$ for some $k$.
Since
\begin{align*}
1-(\beta_{k+1}+1)/p_{k+1}&=2-\alpha-(\beta_k+1)/p_{k+1}\\
&=2-\alpha
+1-(\beta_k+1)/p_{k}-(\beta_k+1)\varepsilon_k
\end{align*}
and $\alpha<2$, the procedure indeed stops in finite steps, i.e.,
$$
1-(\beta_{k}+1)/p_{k}>0
$$
for a finite $k \in \mathbb N$ provided that $\varepsilon_k\le (2-\alpha)/(2(\beta_k+1))$.
Note that to apply the weighted Hardy inequality in each step, we require
$$
(\beta_k+1)/p_{k+1}<1,
$$
which is guaranteed because
\begin{align*}
(\beta_k+1)/p_{k+1}&=(\beta_k+1)/p_{k}-1+(\beta_k+1)\varepsilon_k\\
&=(\beta_{k-1}+1)/p_{k}+\alpha-2+(\beta_k+1)\varepsilon_k\\
&< (\beta_{k-1}+1)/p_{k} <1/2<1.
\end{align*}
Therefore, \eqref{eq12.20} and thus \eqref{eq8.10a} hold for any $\alpha\in (0,2)$.

Next, since $D_{x'}u$ and $u_t$ satisfy the same equation as $u$, from \eqref{eq8.10a} and  \eqref{eq8.11}, we get
\begin{equation}
                        \label{eq8.10b}
\|D_{x'}u\|_{C^{1,1}(Q_{1/2}^+)}
\le N \|D_{x'}u\|_{L_2(Q_1^+)},\quad
\|u_t\|_{C^{1,1}(Q_{1/2}^+)}
\le N \|u\|_{L_2(Q_1^+)}.
\end{equation}
Since
$$
U_t=\overline{a}_{dj}(x_d)D_j u_t,\quad D_{x'}U=\overline{a}_{dj}(x_d)D_j D_{x'}u,
$$
using \eqref{eq8.10b} and the Poincar\'e inequality,
we get
\begin{equation}
                    \label{eq2.10}
\|U_t\|_{L_\infty(Q_{1/2}^+)}+\|D_{x'}U\|_{L_\infty(Q_{1/2}^+)}
\le N \|Du\|_{L_2(Q_1^+)}.
\end{equation}
Furthermore, in view of \eqref{eq8.50}, \eqref{eq8.10a}, \eqref{eq8.10b}, \eqref{eq6.31}, and the zero Dirichlet boundary condition, we have
\begin{equation}
                    \label{eq2.23}
\|D_dU\|_{L_\infty(Q_{1/2}^+)}
\le N \|Du\|_{L_2(Q_1^+)}
\end{equation}
when $\alpha\in (0,1]$. When $\alpha\in (1,2)$,
\begin{equation*}
|D_dU| \le N \|Du\|_{L_2(Q_1^+)}x_d^{1-\alpha}\quad\text{in}\,\,Q_{1/2}^+,
\end{equation*}
which implies that
\begin{align}
                            \label{eq2.29}
|U(t,x',x_d)-U(t,x',y_d)|&\le N\|Du\|_{L_2(Q_1^+)}|x_d^{2-\alpha}-y_d^{2-\alpha}|\notag\\
&\le N\|Du\|_{L_2(Q_1^+)}|x_d-y_d|^{2-\alpha}
\end{align}
for any $(t,x',x_d),(t,x',y_d)\in Q_{1/2}^+$. Combining \eqref{eq8.10a}, \eqref{eq8.10b}, \eqref{eq2.10}, \eqref{eq2.23}, and \eqref{eq2.29} gives \eqref{eq8.10}.

Finally, we show \eqref{eq8.16}. In view of \eqref{eq8.10b} and because $\alpha<2$, it suffices to bound the H\"older semi-norm of $\sqrt\lambda ux_d^{-\alpha/2}$ in $x_d$. For any $(t,x',x_d),(t,x',y_d)\in Q_{1/2}^+$, let
$$
I:=\sqrt \lambda|u(t,x',x_d) x_d^{-\alpha/2}-u(t,x',y_d) y_d^{-\alpha/2}|.
$$
Without loss of generality, we may assume that $0\le x_d<y_d\le 1/2$. When $|x_d-y_d|> |y_d|/4$, by \eqref{eq8.10} and \eqref{eq6.31} we have,
\begin{align*}
I&\le \sqrt \lambda \|Du\|_{L_\infty(Q_{1/2}^+)}(x_d^{1-\alpha/2}+y_d^{1-\alpha/2})\notag\\
&\le N\|u\|_{L_2(Q_{1}^+)}y_d^{1-\alpha/2}
\le N\|Du\|_{L_2(Q_{1}^+)}|x_d-y_d|^{1-\alpha/2},
\end{align*}
where in the last inequality we used the Poincar\'e inequality.
When $|x_d-y_d|\le |y_d|/4$, we have $x_d\in [3y_d/4,y_d)$.
By the mean value theorem, \eqref{eq8.10}, and \eqref{eq6.31}, there exists $s \in (x_d,y_d)$ such that
\begin{align*}
I&= \sqrt \lambda |x_d-y_d| | D_du(t,x',s) s^{-\alpha/2}  - (\alpha/2) u(t,x',s) s^{-1-\alpha/2}|\\
&\le N \sqrt \lambda |x_d-y_d|\|Du\|_{L_\infty(Q_{1/2}^+)}x_d^{-\alpha/2}
\le N\|Du\|_{L_2(Q_{1}^+)}|x_d-y_d|^{1-\alpha/2}.
\end{align*}
This completes the proof of \eqref{eq8.16}.
The proposition is proved.
\end{proof}

\subsection{Interior H\"older estimates for homogeneous equations}  \label{int-schauder}
We fix a point $z_0 = (t_0, x_0) \in \Omega_T$, where $x_0 = (x_0', x_{0d}) \in \bR^{d-1} \times \bR_+$.
Suppose that $\rho \in (0, x_{0d})$, and $\beta \in (0,1)$, we  define the  weighted $\beta$-H\"{o}lder semi-norm of a function $u$ on $Q_\rho(z_0)$ by
\[
\begin{split}
\llbracket u\rrbracket_{C^{\beta/2, \beta}_{\alpha}(Q_\rho(z_0))} = \sup \Big\{ & \frac{|u(s,x) - u(t, y)|}{\big(x_{0d}^{-\alpha/2}|x-y| + |t-s|^{1/2}\big)^{\beta}}: (s,x) \not=(t,y) \\
& \text{and } (s,x), (t,y) \in Q_\rho(z_0)  \Big\}.
\end{split}
\]
As usual, we denote the corresponding weighted norm by
\[
\|u\|_{C^{\beta/2,  \beta}_{\alpha}(Q_\rho(z_0))} = \|u\|_{L_\infty(Q_\rho(z_0))} + \llbracket u\rrbracket_{C^{\beta/2, \beta}_{\alpha}(Q_\rho(z_0))}.
\]
The following result on the interior H\"older estimates of solutions to the homogeneous equation \eqref{x-d.eqn} is needed in the paper.

\begin{proposition}\label{inter-pointwise-est}
Let $z_0 = (t_0, x_0) \in \Omega_T$ and $\rho \in (0, x_{0d}/4)$,  where $x_0 = (x_0', x_{0d}) \in \bR^{d-1} \times \bR_+$.
Suppose that \eqref{c-mu.cond},  \eqref{a-c.cond}, and \eqref{elli-cond} are satisfied on $$
(x_{0d}-r(2\rho, x_{0d}), x_{0d}+r(2\rho, x_{0d})).
$$
If $u \in \sH_{2}^1(Q_{2\rho}(z_0))$ is a weak solution of
\[
\sL_0 u =0 \quad \text{in} \quad Q_{2\rho}(z_0),
\]
then we have
\[
\begin{split}
&  \|x_d^{-\alpha/2} u\|_{L_\infty(Q_{\rho}(z_0))} +   \rho ^{(1-\alpha/2)/2} \llbracket x_d^{-\alpha/2} u\rrbracket_{C^{1/4, 1/2}_\alpha(Q_{\rho}(z_0))} \\
 &  \leq N \left(\fint_{ Q_{2\rho}( z_0)} |x_{d}^{-\alpha/2}u|^{2}dz \right)^{1/2}
\end{split}
\]
and
\[
\begin{split}
& \|D_{x'}u\|_{L_\infty(Q_{\rho}(z_0))} + \|U\|_{L_{\infty}(Q_{\rho}(z_0))} \\
& + \rho^{(1-\alpha/2)/2} \big(\llbracket D_{x'}u\rrbracket_{C^{1/4, 1/2}_\alpha(Q_{\rho}(z_0))} + \llbracket U\rrbracket_{C^{1/4, 1/2}_{\alpha}(Q_{\rho}(z_0))}\big) \\
& \leq N \left(\fint_{ Q_{2\rho}( z_0)} |Du|^{2}dz \right)^{1/2},
\end{split}
\]
where $N = N(\nu, d,\alpha)>0$ and $U = \overline{a}_{di}(x_d)D_iu$.
\end{proposition}

\begin{proof}
By \eqref{Q.def} and as $4\rho < x_{0d}$, we have
\[
r(2\rho,x_{0d}) = \max\{2\rho, x_{0d}\}^{\alpha/2} (2\rho)^{1-\alpha/2} = (2\rho)^{1-\alpha/2} x_{0d}^{\alpha/2}
\]
and
\[
Q_{2\rho}(z_0) = (t_0 - (2\rho)^{2-\alpha} , t_0) \times B_{(2\rho)^{1-\alpha/2} x_{0d}^{\alpha/2}}(x_0).
\]
Let us denote the standard parabolic cylinder centered at $z_0$ with radius
$\rho$ by
\[ \tilde{Q}_\rho(z_0) = (t_0 -\rho^2, t_0) \times B_\rho(x_0) \quad \text{and} \quad \tilde{Q}_\rho = \tilde{Q}_\rho(0). \]
Also, let
\[ v(t,x) = u(\rho^{2-\alpha}t + t_0,  \rho^{1-\alpha/2} x_{0d}^{\alpha/2}  x + x_0), \quad (t,x) \in \tilde{Q}_{2}. \]
We then see that $v$ is a weak solution of
\begin{equation} \label{v-heat-eqn}
\tilde{\mu}(x_d) v_t + \lambda \rho^{2-\alpha}\tilde{c}_0(x_d) v  -  D_i (\tilde{a}_{ij}(x_d) D_j v) =0 \quad \text{in} \quad \tilde{Q}_{2},
\end{equation}
where
\[
\begin{split}
& \tilde{a}_{ij}(x_d) = a_{ij}(\rho^{1-\alpha/2} x_{0d}^{\alpha/2}  x_d + x_{0d}), \\
& \tilde{c}_0(x_d) = x_{0d}^{\alpha} \, \overline{c}_0(\rho^{1-\alpha/2} x_{0d}^{\alpha/2}  x_d + x_{0d})\big[\mu(\rho^{1-\alpha/2} x_{0d}^{\alpha/2}  x_d + x_{0d})\big]^{-1}, \\
&\tilde{\mu}(x_d) = x_{0d}^{\alpha} \big[\mu(\rho^{1-\alpha/2} x_{0d}^{\alpha/2}  x_d + x_{0d})\big]^{-1}.
\end{split}
\]
Due to this and \eqref{c-mu.cond} and as $\rho/x_{0d} <1/4$, we see that
\[
\mu(\rho^{1-\alpha/2} x_{0d}^{\alpha/2}  x_d + x_{0d}) \sim x_{0d}^{\alpha}[(\rho /x_{0d})^{1-\alpha/2} x_d +1]^{\alpha} \sim x_{0d}^{\alpha} \quad \text{for all} \,\, |x_d| <2.
\]
Therefore, there is a constant $N_0 = N_0(\nu, \alpha)  \in (0,1)$ such that
\[
 N_0 \leq \tilde{\mu}(x_d), \tilde{c}_0(x_d) \leq N_0^{-1}, \quad \forall\, z = (t, x', x_d) \in \tilde{Q}_{2}.
\]
Consequently, the coefficients in \eqref{v-heat-eqn} are uniformly elliptic and bounded  in $\tilde{Q}_{2}$.
Then, adapting the proof of H\"{o}lder estimates  in \cite[Lemma 3.5]{Dong-Kim11}  to \eqref{v-heat-eqn},  we obtain
\[
\begin{split}
\|v\|_{C^{1/4, 1/2}(\tilde{Q}_{1} )} & \leq N \left(\fint_{\tilde{Q}_{2}} |v|^{2}dz \right)^{1/2} = N \left(\fint_{Q_{2\rho}( z_0)} |u|^{2}dz \right)^{1/2} \\
& \leq N  x_{0d}^{\alpha/2} \left(\fint_{Q_{2\rho}( z_0)} |x_d^{-\alpha/2}u|^{2}dz \right)^{1/2},
\end{split}
\]
where in the last step, we use the fact that $x_d \sim x_{0d}$ for all $z= (z', x_d) \in Q_{2\rho}(z_0)$. Now, for $(s,x)$ and $(\tau, y) \in  Q_{1}$ with $(s,x) \not= (\tau, y)$, we have
\[
  \frac{|v(s, x)-v(\tau, y)|}{\big( |x-y| + |s-\tau|^{1/2} \big)^{1/2}}  = \frac{\rho^{(1-\alpha/2)/2} |u(s', \hat{x})-u(\tau', \hat{y})|}{\big( x_{0d}^{-\alpha/2} |\hat{x}-\hat{y}| + |s'-\tau'|^{1/2}\big)^{1/2}},
\]
where
\[
\begin{split}
&\hat{x} = \rho^{1-\alpha/2} x_{0d}^{\alpha/2}  x + x_0, \quad \hat{y} = \rho^{1-\alpha/2} x_{0d}^{\alpha/2}  y + x_0,\\
&s' = \rho^{2-\alpha} s + t_0, \quad \tau' = \rho^{2-\alpha} \tau + t_0,
\end{split}
\]
which implies that
\[
 \rho^{(1-\alpha/2)/2} \llbracket u\rrbracket_{C^{1/4, 1/2}_{\alpha}(Q_{\rho}(z_0))} =  \llbracket v\rrbracket_{C^{1/4, 1/2}(\tilde{Q}_{1} )}.
\]
Therefore,
\begin{equation} \label{u-hold.06}
\begin{split}
& \|u\|_{L_\infty(Q_{\rho}(z_0))} + \rho ^{(1-\alpha/2)/2} \llbracket u\rrbracket_{C^{1/4, 1/2}_{\alpha}(Q_{\rho}(z_0))} \\
& \leq N  x_{0d}^{\alpha/2} \left(\fint_{Q_{2\rho}( z_0)} |x_d^{-\alpha/2}u|^{2}dz \right)^{1/2}.
\end{split}
\end{equation}
Now, for $(t, x), (s,y) \in Q_{2\rho}( z_0)$ with $x=(x',x_d)$ and $y = (y', y_d)$,  by the triangle inequality, we have
\[
\begin{split}
& |x_d^{-\alpha/2}u(t,x) -y_d^{-\alpha/2}u(s,y)|\\
& \leq  |u(t,x) -u(s,y)| \, x_d^{-\alpha/2} + |x_d^{-\alpha/2} -y_{d}^{-\alpha/2}| \, |u(s,y)| \\
& \leq N(\alpha)\, x_{0d}^{-\alpha/2}\,\big( |u(t,x) -u(s,y)| + |x_d-y_d| \, x_{0d}^{-1} \, \|u\|_{L_\infty(Q_{\rho}(z_0))} \big) \\
& \leq Nx_{0d}^{-\alpha/2}\big(x_{0d}^{-\alpha/2} |x-y| +  |t-s|^{1/2}\big)^{1/2} \\
&\quad \cdot\big(\llbracket u\rrbracket_{C^{1/4, 1/2}_{\alpha}(Q_{\rho}(z_0))}+  |x_d-y_d|^{1/2} \, x_{0d}^{\alpha/4-1} \, \|u\|_{L_\infty(Q_{\rho}(z_0))} \big) \\
& \leq Nx_{0d}^{-\alpha/2}\big(x_{0d}^{-\alpha/2}  |x-y| +  |t-s|^{1/2} \big)^{1/2}\\
&\quad \cdot \big(\llbracket u\rrbracket_{C^{1/4, 1/2}_{\alpha}(Q_{\rho}(z_0))} + \rho^{(1-\alpha/2)/2} x_{0d}^{\alpha/2-1} \|u\|_{L_\infty(Q_{\rho}(z_0))}\big),
\end{split}
\]
where we used the fact that $x_d, y_d \sim x_{0d}$ in the second  inequality and $|x_d - y_d| \leq N \rho^{1-\alpha/2} x_{0d}^{\alpha/2}$ in the last inequality.
Therefore, as $\rho/x_{0d} \leq 1/4$ and \eqref{u-hold.06}, we obtain
\[
\begin{split}
& \|x_d^{-\alpha/2} u\|_{L_\infty(Q_{\rho}(z_0))} + \rho^{(1-\alpha/2)/2} \llbracket x_d^{-\alpha/2}u\rrbracket_{C^{1/4, 1/2}_{\alpha}(Q_{\rho}(z_0))} \\
& \leq   N  \left(\fint_{Q_{2\rho}( z_0)} |x_d^{-\alpha/2}u|^{p_0}dz \right)^{1/p_0}
\end{split}
\]
and this proves the first assertion  of the proposition.

Next, we prove the second assertion.
Again,  adapting the proof of  \cite[Lemma 3.5]{Dong-Kim11} to the equation \eqref{v-heat-eqn}, we see  that
\[
 \|D_{x'}v\|_{C^{1/4, 1/2} (\tilde{Q}_{1})} + \|V\|_{C^{1/4, 1/2} (\tilde{Q}_{1})}  \leq N(\nu, d) \left(\fint_{\tilde{Q}_{2}} |Dv|^{2}dz \right)^{1/2},
\]
where $V = \tilde{a}_{dj}(x_d) D_j v$.
Then, by scaling back as before, we obtain the second assertion of the proposition.
The proof is completed.
\end{proof}

\subsection{Mean oscillation estimates and proof of Theorem \ref{sim-eqn-thrm}} \label{oss-est-simple-coffe}

We next prove the following mean oscillation estimates of weak solutions to homogeneous equations.

\begin{lemma} \label{hom-osc-est}
Let $z_0 =(z_0', x_{0d}) \in \overline{\Omega}_T$ and $\rho>0$. Assume that $u \in \sH_2^1(Q_{14 \rho}^+(z_0))$ is a weak solution of
\[
\sL_0 u= 0 \quad \text{in} \quad Q_{14 \rho}^+(z_0)
\]
with the boundary condition $u =0$ on $\{x_d =0\} \cap Q_{14\rho}(z_0)$ if $\{x_d =0\} \cap \overline{Q_{14\rho}(z_0)}$ is not empty.
Then, for every $\kappa \in (0,1)$,
\[ \big(|v - (v)_{Q_{\kappa \rho}^+(z_0)}|\big)_{Q_{\kappa\rho}^+(z_0)} \leq N \kappa^{\gamma_0}\left[\big(|v|^2\big) ^{1/2}_{Q_{14\rho}^+(z_0)} +\big(|Du|^2\big) ^{1/2}_{Q_{14\rho}^+(z_0)}  \right]  \]
for $v = \sqrt{\lambda} x_d^{-\alpha/2} u$, and
\[
\begin{split}
&  \big(|D_{x'}u - (D_{x'}u)_{Q_{\kappa \rho}^+(z_0)}|\big)_{Q_{\kappa\rho}^+(z_0)} + \big(|U - (U)_{Q_{\kappa \rho}^+(z_0)}|\big)_{Q_{\kappa\rho}^+(z_0)} \\
&  \leq N \kappa^{\gamma_0} \big(|Du|^2\big)^{1/2}_{Q_{14\rho}^+(z_0)},
\end{split}
\]
where $\gamma_0 = \min\{1, 2-\alpha\}/4$, $U = \overline{a}_{dj}D_ju$, and $N = N(d, \nu, \alpha)>0$.
\end{lemma}

\begin{proof}
By a scaling argument, without loss of generality, we can assume that $\rho=1$.
We consider two cases.

\medskip

\noindent
{\bf Case 1:} $x_{0d} \leq 4$.
Let $\tilde{z}_0 = (z_0',0)$, and it follows from \eqref{Q.def} that
\[ Q_{1}^+(z_0) \subset Q_{5}^+(\tilde{z}_0) \subset Q_{10}^+(\tilde{z}_0) \subset Q_{14}^+(z_0).
\]
Then, it follows from the mean value theorem and Proposition \ref{prop3} that
\[
\begin{split}
& \big(|D_{x'}u - (D_{x'}u)_{Q_{\kappa}^+(z_0)}|\big)_{Q_{\kappa}^+(z_0)} \\
&   \leq N(d) \kappa \big[ \|DD_{x'}u\|_{L_\infty(Q_{1}^+(z_0))} + \|D_{x'}u_t\|_{L_\infty(Q_{1}^+(z_0))} \big]\\
&  \leq  N  \kappa \|D_{x'}u\|_{C^{1,1}(Q_{5}^+(\tilde{z}_0))} \leq N  \kappa \big(|Du|^2\big)_{Q_{10}^+(\tilde{z}_0)}^{1/2} \\
& \leq N  \kappa \big(|Du|^2\big)_{Q_{14}^+( z_0)}^{1/2}.
\end{split}
\]
Recall that $\gamma = \min\{1, 2-\alpha\}$.
By a similar argument,
\[
\begin{split}
\big(|U - (U)_{Q_{\kappa }^+(z_0)}|\big)_{Q_{\kappa}^+(z_0)}  &\leq N \kappa^{2-\alpha} \|\partial_t U\|_{L_\infty(Q_{1}^+( z_0))} + \kappa^{\gamma} \llbracket U\rrbracket_{C^{0, \gamma} (Q_{1}^+(z_0))}\\
& \leq N \kappa^{\gamma} \big(|Du|^2\big)_{Q_{14}^+( z_0)}^{1/2}.
\end{split}
\]
Finally, we write $v = \sqrt{\lambda}x_d^{-\alpha/2} u$.
Applying the mean value theorem and Proposition \ref{prop3}, we get
\[
\begin{split}
& \big(|v - (v)_{Q_{\kappa }^+(z_0)}|\big)_{Q_{\kappa}^+(z_0)} \leq N\kappa^{1-\alpha/2} \|v\|_{C^{1, 1-\alpha/2}(Q_{5}^+(\tilde{z}_0) )} \\
& \leq N \kappa^{1-\alpha/2} \big(|Du|^2\big)^{1/2}_{Q_{10}^+(\tilde{z}_0)} \leq N \kappa^{1-\alpha/2} \big(|Du|^2\big)^{1/2}_{Q_{14}^+(z_0)}.
\end{split}
\]
Then, the desired inequalities follow as $\kappa \in (0,1)$.

\medskip

\noindent
{\bf Case 2:}  $x_{0d}>4$.
The proof is similar to  Case 1,  instead we apply Proposition \ref{inter-pointwise-est}.
For example, for $v = \sqrt{\lambda}x_d^{-\alpha/2} u$, we have
\[
\begin{split}
& \big(|v - (v)_{Q_{\kappa}^+(z_0)}|\big)_{Q_{\kappa}^+(z_0)}  \leq N\kappa^{1/2 -\alpha/4} \llbracket v\rrbracket_{C^{1/4, 1/2}_{\alpha}(Q_{1}^+(z_0))} \\
& \leq N \kappa^{1/2 -\alpha/4} \left(\fint_{Q_{2}^+(z_0)} |v(z)|^2 dz \right)^{1/2} \leq N \kappa^{1/2 -\alpha/4} \left(\fint_{Q_{14}^+(z_0)} |v(z)|^2 dz \right)^{1/2},
\end{split}
\]
where we used the doubling properties of the measure.
The oscillation estimates of $D_{x'}v$ and $U$ can be proved in the same way.
\end{proof}

Next, we prove the following proposition on the oscillation estimates for weak solution of the non-homogeneous equation \eqref{x-d.eqn}.

\begin{proposition}[Mean oscillation estimates] \label{osc-prop}
Assume that $F \in L_{2,\textup{loc}}(\Omega_T)^d$ and $f= f_1 + f_2$ such that $x_d^{1-\alpha} f_1$ and $x_d^{-\alpha/2} f_2$ are in $L_{2,\textup{loc}}(\Omega_T)$.
If $u \in \sH_{2, \textup{loc}}^1(\Omega_T)$ is a weak solution of \eqref{x-d.eqn}, then for every $z_0 \in \overline{\Omega}_T$,  $\rho \in (0, \infty)$, and $\kappa \in (0,1)$,
\[
\begin{split}
 \big(|v - (v)_{Q_{\kappa \rho}^+(z_0)}| \big)_{Q_{\kappa\rho}^+(z_0)} & \leq  N \kappa^{\gamma_0}\big[ \big(|v|^2\big)_{Q_{14 \rho}^+(z_0)}^{1/2}+ \big(|Du|^2\big)_{Q_{14\rho}^+(z_0)}^{1/2}\big]\\
& + N \kappa^{-\gamma_1} \big[ \big(|F|^2\big)_{Q_{14\rho}^+(z_0)}^{1/2} +   \big(|g|^2\big)_{Q_{14\rho}^+(z_0)}^{1/2} \big]
\end{split}
\]
and
\[
\begin{split}
 \big(|\cU - (\cU)_{Q_{\kappa \rho}^+(z_0)}| \big)_{Q_{\kappa\rho}^+(z_0)} & \leq  N \kappa^{\gamma_0} \big(|Du|^2\big)_{Q_{14\rho}^+(z_0)}^{1/2}\\
& + N \kappa^{-\gamma_1} \big[ \big(|F|^2\big)_{Q_{14\rho}^+(z_0)}^{1/2} +   \big(|g|^2\big)_{Q_{14\rho}^+(z_0)}^{1/2} \big],
\end{split}
\]
where $v = \sqrt{\lambda} x_d^{-\alpha/2}u$, $\cU = (D_{x'}u, U)$ with $U = \overline{a}_{di}(x_d)D_iu$, $g = x_d^{1-\alpha} |f_1|+ \lambda^{-1/2}  x_d^{-\alpha/2} |f_2|$, $\gamma_0 = \min\{1, 2-\alpha\}/ 4$, $\gamma_1 = (d+2 -\alpha)/2$,  and $N = N(d, \nu, \alpha)>0$.
\end{proposition}

\begin{proof}
Let $w \in \sH_2^1(\Omega_T)$ be a weak solution of
\[
\sL_0 w = \mu(x_d) D_i(F_i \chi_{Q_{14\rho}^+(z_0)}(z)) + f \chi_{Q_{14\rho}^+(z_0)}(z) \quad \text{in} \quad \Omega_T
\]
with the boundary condition $w =0$ on $\{x_d =0\}$.
The existence of such solution is guaranteed by Theorem \ref{L-2.theorem}.
By the same theorem, we have
\begin{equation} \label{w.est-1}
\begin{split}
& \|Dw\|_{L_2(\Omega_T)} + \sqrt{\lambda}\|x_d^{-\alpha/2} w\|_{L_2(\Omega_T)} \leq N \|F\|_{L_2(Q_{14\rho}^+(z_0))} + N\|g\|_{L_2(Q_{14\rho}^+(z_0))}.
\end{split}
\end{equation}
Next, note that $h=u-w \in \sH_2^1(Q_{14\rho}^+(z_0))$ is a weak solution of
\[
\sL_0 h =0 \quad \text{in} \quad Q_{14\rho}^+(z_0)
\]
with the boundary condition $h =0$ on $\{x_d =0\} \cap \overline{Q_{14}(z_0)}$.
Denote
$$
\mathcal{W} =  (D_{x'}w, \overline{a}_{di}D_iw)\quad \text{and}\quad
\mathcal{H} = (D_{x'}h, \overline{a}_{di}D_ih).
$$
Then, applying Lemma \ref{hom-osc-est}, we obtain
\begin{equation} \label{h.oss.est}
\big(|\mathcal{H} - (\mathcal{H})_{Q_{\kappa \rho}^+(z_0)}| \big)_{Q_{\kappa\rho}^+(z_0)}  \leq N \kappa^{\gamma_0} \big(|Dh|^2 \big)_{Q_{14 \rho}^+(z_0)}^{1/2}.
\end{equation}
Moreover,
\begin{equation} \label{tild-h.0705}
\big(|\tilde{h} - (\tilde{h})_{Q_{\kappa \rho}^+(z_0)}| \big)_{Q_{\kappa\rho}^+(z_0)}  \leq N \kappa^{\gamma_0}\big[ \big(|\tilde{h}|^2 \big)_{Q_{14 \rho}^+(z_0)}^{1/2}  +  \big(|Dh|^2 \big)_{Q_{14 \rho}^+(z_0)}^{1/2} \big]
\end{equation}
with $\tilde{h} = \lambda^{1/2}x_d^{-\alpha/2} h$.
By the triangle inequality, H\"older's inequality, and \eqref{h.oss.est}, we have
\begin{align} \notag
& \big(|\cU - (\cU)_{Q_{\kappa \rho}^+(z_0)}| \big)_{Q_{\kappa\rho}^+(z_0)} \\ \notag
& \leq  \big(|\mathcal{H} - (\mathcal{H})_{Q_{\kappa \rho}^+(z_0)}| \big)_{Q_{\kappa\rho}^+(z_0)} + \big(|\mathcal{W} - (\mathcal{W})_{Q_{\kappa \rho}^+(z_0)}| \big)_{Q_{\kappa\rho}^+(z_0)} \\ \notag
& \leq \big(|\mathcal{H} - (\mathcal{H})_{Q_{\kappa \rho}^+(z_0)}| \big)_{Q_{\kappa\rho}^+(z_0)} + N(d) \kappa^{-\gamma_1}(|\mathcal{W}|^2)_{Q_{14\rho}^+(z_0)}^{1/2} \\ \notag
& \leq N \kappa^{\gamma_0}\big(|Dh|^2 \big)_{Q_{14 \rho}^+(z_0)}^{1/2} +  N(d) \kappa^{-\gamma_1}(|Dw|^2)_{Q_{14\rho}^+(z_0)}^{1/2}  \\ \label{U-ose.est}
& \leq N \big[ \kappa^{\gamma_0}\big(|Du|^2 \big)_{Q_{14 \rho}^+(z_0)}^{1/2} + \kappa^{-\gamma_1}(|Dw|^2)_{Q_{14\rho}^+(z_0)}^{1/2} \big],
\end{align}
where we used $\kappa \in (0,1)$ and the following fact from \eqref{Q.def} and \eqref{r.def} that
\begin{equation} \label{compared-Q}
\frac{|Q_{14\rho}^+(z_0)|}{|Q_{\kappa\rho}^+(z_0)|} = N(d) \kappa^{-2+\alpha}\Big[ \frac{ r(14\rho, x_{0d})}{r(\kappa \rho, x_{0d})} \Big]^d \leq N(d) \kappa^{-2\gamma_1}
\end{equation}
with $\gamma_1 = (d+2 -\alpha)/2$.
Then, by using  \eqref{w.est-1} and \eqref{U-ose.est}, we obtain the desired estimate for $\cU$.
The oscillation estimate for $v= \lambda^{1/2}x_d^{-\alpha/2} u$ can be proved similarly using \eqref{w.est-1} and \eqref{tild-h.0705}.
\end{proof}

\begin{proof}[Proof of Theorem \ref{sim-eqn-thrm}]
We consider the cases when $p>2$ and $p \in (1,2)$ as the case when $p=2$ was proved in Theorem \ref{L-2.theorem}.

\medskip

\noindent
{\bf Case 1:} $p>2$.
We prove the a priori estimate \eqref{apr-est} assuming that $u \in \sH_p^1(\Omega_T)$.
Let $v$ and $\cU$ be defined as in Proposition \ref{osc-prop}.
Using Proposition \ref{osc-prop}, we have
\[
\cU^{\#}  \leq N \big[ \kappa^{\gamma_0}\cM(|Du|^{2})^{1/2}   + \kappa^{-\gamma_1} \cM(|F|^2)^{1/2} + \kappa^{-\gamma_1} \cM(|g|^2)^{1/2}\big]
\]
and
\[
v^{\#}   \leq N \kappa^{\gamma_0}\big(\cM(|v|^{2})^{1/2} + \cM(|Du|^{2})^{1/2} \big)+ N \kappa^{-\gamma_1}  \big (\cM(|F|^2)^{1/2} +   \cM(|g|^2)^{1/2}\big)
\]
in $\Omega_T$, where   $g = x_d^{1-\alpha} |f_1| + \lambda^{-1/2} x_d^{-\alpha/2}|f_2|$, $\cU^{\#}$ and $v^{\#}$ are the Fefferman-Stein sharp functions of $\cU$ and $v$, respectively, and $\cM$ is the  Hardy-Littlewood maximal operator defined by using the quasi-metric constructed in Section \ref{filtration.def}.
Recall that $\cU$ and $|Du|$ are comparable.
We now apply the Fefferman-Stein theorem and Hardy-Littlewood maximal function theorem (see, for instance, \cite[Sec. 3.1-3.2]{Krylov}) to obtain
\[
\begin{split}
\|Du\|_{L_p(\Omega_T)} + \sqrt{\lambda} \|x_d^{-\alpha/2}u\|_{L_p(\Omega_T)} & \leq N \Big[\kappa^{\gamma_0}\big(\sqrt{\lambda} \|x_d^{-\alpha/2}u\|_{L_p(\Omega_T)} + \|Du\|_{L_p(\Omega_T)} \big) \\
& + \kappa^{-\gamma_1} \|F\|_{L_p(\Omega_T)} + \kappa^{-\gamma_1} \|g\|_{L_p(\Omega_T)}  \Big],
\end{split}
\]
where $N = N(d, \nu, \alpha, p)>0$ and we used $p>2$.
From this, and by choosing $\kappa \in (0,1)$ sufficiently small, we obtain
\[
\|Du\|_{L_p(\Omega_T)} + \sqrt{\lambda} \|x_d^{-\alpha/2}u\|_{L_p(\Omega_T)} \leq N \Big[ \|F\|_{L_p(\Omega_T)} +  \|g\|_{L_p(\Omega_T)}  \Big].
\]
Then,  \eqref{apr-est}  is proved.

Note that  \eqref{apr-est} implies the uniqueness of solutions in $\sH_p^1(\Omega_T)$.
Therefore, it remains to show the existence of solutions.
We first consider the special case when $F,f_1,f_2\in C_0^\infty(\Omega_T)$.
In this case, by Theorem \ref{L-2.theorem}, there is a unique solution $u\in \sH_2^1(\Omega_T)$ to \eqref{x-d.eqn}.
Since $F$ and $f$ are smooth and compactly supported, we can modify the proof of Proposition \ref{prop3} to get
\begin{equation}
                    \label{eq2.49}
\|Du\|_{L_\infty(Q^+_{1/2}(z_0))}+\|v\|_{L_\infty(Q^+_{1/2}(z_0))}\le N\|Du\|_{L_2(Q^+_{1}(z_0))}+C_{F,f}(z_0)
\end{equation}
for any $z_0\in \partial\Omega_\infty\cap\{t<T\}$, where the constant $C_{F,f}(z_0)$ vanishes when $|z_0|$ is sufficiently large.
A similar estimate holds in the interior of the domain:
\begin{align}
                    \label{eq2.49b}
&\|Du\|_{L_\infty(Q^+_{1/2}(z_0))}+\|v\|_{L_\infty(Q^+_{1/2}(z_0))}\notag\\
&\le N x_{0d}^{- \alpha d/4}\big\||Du|+x_{0d}^{-\alpha/2}|u|\big\|_{L_2(Q^+_{1}(z_0))}+C_{F,f}(z_0)
\end{align}
for any $z_0\in \Omega_T$ satisfying $|x_{0d}|\ge 1/2$.
From \eqref{eq2.49} and \eqref{eq2.49b}, we see that $Du$ and $v$ are bounded in $\Omega_T$, which together with the equation \eqref{x-d.eqn} implies that $u\in \sH_p^1(\Omega_T)$.
Finally, for general $F$ and $f$, we take sequences of functions $\{F^{(n)}\}$, $\{f_1^{(n)}\}$, and $\{f_2^{(n)}\}$   in $C_0^\infty(\Omega_T)$ such that
$$
F^{(n)}\to F,\quad x_d^{1-\alpha}f_1^{(n)}\to x_d^{1-\alpha}f_1,\quad x_d^{-\alpha/2}f_2^{(n)}\to x_d^{-\alpha/2}f_2
$$
in $L_p(\Omega_T)$.
From the proof above, for each $n\in \bN$ there is a unique solution $u^{(n)}\in \sH_p^1(\Omega_T)$ to the equation \eqref{x-d.eqn} with $F^{(n)}$, $f_1^{(n)}$, and $f_2^{(n)}$ in place of $F$, $f_1$ and $f_2$.
By using the a priori estimate \eqref{apr-est}, we see that  $\{Du^{(n)}\}$ and $\{\sqrt\lambda x_d^{-\alpha/2}u^{(n)}\}$ are Cauchy sequences in $L_p(\Omega_T)$.
After passing to the limit, we then obtain a solution $u\in \sH_p^1(\Omega_T)$ to \eqref{x-d.eqn}.

\medskip

\noindent
{\bf Case 2:} $p \in (1,2)$.
As before, we first prove  \eqref{apr-est}.
We follow the standard duality argument.
Let $q = {p}/(p-1) \in (2,\infty)$, $G =(G_1, G_2,\ldots, G_d)\in L_q(\Omega_T)^d$ and $h = h_1 + h_2$ such that $\tilde{h} = x_d^{1-\alpha}|h_1| + \lambda^{-1/2}x_d^{-\alpha/2}|h_2| \in L_q(\Omega_T)$.
We consider the adjoint problem
\begin{equation} \label{adj-veqn.1}
- \tilde u_t + \lambda \bar c_0 \tilde u - \mu(x_d) D_i\big( \overline{a}_{ji}(x_d) D_{j} \tilde u +G_{i}\chi_{(-\infty, T)}\big)  =  h\chi_{(-\infty, T)}
\end{equation}
in $\bR^{d+1}_+$ with the boundary condition $v = 0$ on $\partial \bR^{d+1}_+$.
By Case 1 and a change of the time variable $t\to -t$, there exists a unique weak solution $\tilde u \in \sH^1_q(\bR \times \bR_+^d)$ of \eqref{adj-veqn.1} and
\begin{equation}\label{eqnv-est-06}
\begin{split}
& \int_{\bR^{d+1}_+} \big(|D\tilde u(z)|^q + \lambda^{q/2} |x_d^{-\alpha/2}\tilde u(z)|^q \big) \,dz \leq N\int_{\Omega_T} \big(|G(z)|^q + |\tilde{h}(z)|^q \big) \, dz.
\end{split}
\end{equation}
Note also $\tilde u =0$ for $t \geq T$ because of the uniqueness of solutions to \eqref{adj-veqn.1}.
Then, as in Definition \ref{weak-sol-def}, we test \eqref{x-d.eqn} with $v$ and test \eqref{adj-veqn.1} with $u$.
We  then obtain
\begin{align} \notag
& \int_{\Omega_T}\big(G(z)\cdot D u(z)  -  \mu(x_d)^{-1} h(z) u(z) \big)\, dz \\ \label{uv-dual.06}
&= \int_{\Omega_T}\big(F(z)\cdot D \tilde u (z) -  \mu(x_d)^{-1} f(z) \tilde u(z) \big)\, dz.
\end{align}
We next control the terms on the right-hand side of \eqref{uv-dual.06}.
By H\"{o}lder's inequality, and \eqref{eqnv-est-06}, the first term on the right-hand side of \eqref{uv-dual.06} can be bounded as
\[
\left|\int_{\Omega_T} F(z)\cdot D \tilde u (z) dz \right| \leq  N \|F\|_{L_p(\Omega_T)} \Big[\|G\|_{L_q(\Omega_T)} + \|\tilde{h}\|_{L_q(\Omega_T)} \Big].
\]
To bound the second term  on the right-hand side of \eqref{uv-dual.06}, we use the condition on $\mu$ in \eqref{c-mu.cond},  H\"{o}lder's inequality,  and Hardy's inequality to obtain
\[
\begin{split}
& \left|\int_{\Omega_T} \mu(x_d)^{-1} f(z) \tilde u(z) dz \right| \\
&  \leq N(\nu) \int_{\Omega_T}\big( |x_d^{1-\alpha}f_1(z)| \,  \big| \tilde u/x_d\big|  + | x_d^{-\alpha/2}f_2| \, |x_d^{-\alpha/2}\tilde u|  \big)\, dz  \\
& \leq N(\nu)\Big[ \|x_d^{1-\alpha}f_1\|_{L_p(\Omega_T)} \|\tilde u/x_d\|_{L_q(\Omega_T)} + \|x_d^{-\alpha/2}f_2\|_{L_p(\Omega_T)} \|x_d^{-\alpha/2}\tilde u\|_{L_q(\Omega_T)} \Big] \\
&  \leq N(\nu, d, q) \|g\|_{L_p(\Omega_T)} \Big[  \|D\tilde u\|_{L_q(\Omega_T)} + \lambda^{1/2} \|x_d^{-\alpha/2}\tilde u\|_{L_q(\Omega_T)} \Big]\\
& \leq N\|g\|_{L_p(\Omega_T)}\Big[\|G\|_{L_q(\Omega_T)} + \|\tilde{h}\|_{L_q(\Omega_T)} \Big],
\end{split}
\]
where \eqref{eqnv-est-06} is used in the last inequality and we recall
$$
g = x_d^{1-\alpha} |f_1| + \lambda^{-1/2}x_d^{-\alpha/2}|f_2|.
$$
In summary, it follows from \eqref{uv-dual.06} that
\[
\begin{split}
& \left|\int_{\Omega_T}\big(G(z)\cdot D u(z) -
\mu(x_d)^{-1} h(z)   u (z)\big)\, dz\right| \\
& \leq N\Big(\|F\|_{L_p(\Omega)} + \|g\|_{L_{p}(\Omega_T)}\Big)
\Big( \|G\|_{L_q(\Omega_T)} + \|\tilde{h}\|_{L_q(\Omega_T)} \Big).
\end{split}
\]
Because of the last estimate, the condition \eqref{c-mu.cond} for $\mu$, and as $G$ and $h$ are arbitrary, we obtain the a priori estimate  \eqref{apr-est}.

Now we prove the existence of solutions.
As in Case 1, we only need to consider the case when $F,f_1,f_2\in C_0^\infty(\Omega_T)$.
By Theorem \ref{L-2.theorem}, there is a unique solution $u\in \sH_2^1(\Omega_T)$ to \eqref{x-d.eqn}.
Now we take $G,f_1,f_2\in C_0^\infty(\Omega_T)$. Let $w\in \sH_2^1(\Omega_T)$ be the unique solution to \eqref{adj-veqn.1}.
According to the proof in Case 1, we know that $w\in \sH_q^1(\Omega_T)$.
By the duality argument above, we infer that $Du, v\in L_p(\Omega_T)$ and  \eqref{apr-est} holds.
Therefore, from the equation, we conclude that  $u\in \sH_p^1(\Omega_T)$.
The theorem is proved.
\end{proof}


\section{Proofs of Theorems \ref{main-thrm} and \ref{thrm-2}} \label{sec-5}

In this section, we prove Theorems \ref{main-thrm} and  \ref{thrm-2}.
Recall the definitions of
$$
[a_{ij}]_{14\rho, z'_0}(\cdot)\quad \text{and}\quad
[c_0]_{14\rho, z'_0}(\cdot)
$$
in \textup{Assumption} \ref{mean-osc} \textup{($\rho_0, \delta$)}.
We first prove the following lemma on the oscillation estimates of solutions of \eqref{eq6.12}.

\begin{lemma}
            \label{oss-lem-vmo}
Let $\nu \in (0,1)$, $\alpha \in (0,2)$, $\rho_0 >0$, $\delta>0$, and assume that  \eqref{ellipticity-cond}, \eqref{c-mu.cond}, and \textup{Assumption} \ref{mean-osc} \textup{($\rho_0, \delta$)} are satisfied.
Let $q \in (2,  \infty)$ and suppose that $u \in \sH_{q, \textup{loc}}^1(\Omega_T)$ is a weak solution of \eqref{eq6.12} with $F \in L_{2, \textup{loc}}(\Omega_T)$ and $f = f_1 + f_2$ such that  $g = x_d^{1-\alpha}|f_1| + \lambda^{-1/2} x_d^{-\alpha/2} |f_2| \in L_{2, \textup{loc}}(\Omega_T)$.
Then, there is a constant $N = N(\nu, \alpha, d, q)>0$ such that
\[
\begin{split}
&\big(|\cU - (\cU)_{Q_{\kappa \rho}^+(z_0)}| \big)_{Q_{\kappa \rho}^+(z_0)} +
\big(|v - (v)_{Q_{\kappa \rho}^+(z_0)}| \big)_{Q_{\kappa \rho}^+(z_0)} \\
& \leq N \big( \kappa^{\gamma_0} + \kappa^{-\gamma_1} \delta^{1/2-1/q} \big)\Big[ \big(|v|^q\big)_{Q_{14 \rho}^+(z_0)}^{1/q} + \big(|Du|^q\big)_{Q_{14 \rho}^+(z_0)}^{1/q} \Big]
 \\
& \quad  + N \kappa^{-\gamma_1} \Big[ \big(|F|^2\big)_{Q_{14 \rho}^+(z_0)}^{1/2} + \big(|g|^2\big)_{Q_{14 \rho}^+(z_0)}^{1/2}  \Big]
\end{split}
\]
for every $z_0 \in \overline{\Omega}_T$, $\rho \in (0, \rho_0/14)$, and $\kappa \in (0,1)$, where $\cU = (D_{x'}u, U_{Q_{14 \rho}^+(z_0)})$ with $U_{Q_{14 \rho}^+(z_0)} = [a_{dj}]_{14\rho, z'_0}(x_d) D_ju$, and $v = \lambda^{1/2}x_d^{-\alpha/2}u$.
Here, $\gamma_0 = \min\{1, 2-\alpha\}/4$ and $\gamma_1 = (d+2 -\alpha)/2$.
\end{lemma}

\begin{proof}
We write $z_0' = (t_0, x_0')$.
Let $\tilde F = (\tilde F_1, \tilde F_2,\ldots, \tilde F_d)$ with
\[
 \tilde F_i =\big[ \big(a_{ij} -[a_{ij}]_{14\rho, z'_0}(x_d) \big)D_j u + F_i\big] \chi_{Q_{14\rho}^+(z_0)}(z), \quad i =1, 2,\ldots, d,
\]
so that  $u \in \sH_{p}^1(Q_{14\rho}^+(z_0))$ is a weak solution of
\[
u_t + \lambda [c_0]_{14\rho,z_0'} u - \mu(x_d) D_{i}\big([a_{ij}]_{14\rho, z'_0}(x_d) D_j u + \tilde F_i \big) = \tilde f_1+\tilde f_2 \quad \text{in} \quad Q_{14\rho}^+(z_0)
\]
with the boundary condition $u =0$ on $\{x_d =0\}$,
where
\[
 \tilde f_1 = f_1 \chi_{Q_{14\rho}^+(z_0)}(z), \quad \tilde f_2 =\big[ \lambda \big([c_0]_{14\rho, z'_0}(x_d) - c_0\big) u + f_2\big]\chi_{Q_{14\rho}^+(z_0)}(z).
\]
Then, applying Proposition \ref{osc-prop}, we have
\[
\begin{split}
 \big(|\cU - (\cU)_{Q_{\kappa \rho}^+(z_0)}| \big)_{Q_{\kappa \rho}^+(z_0)} & \leq N \kappa^{\gamma_0} \big(|\cU|^2\big)_{Q_{14 \rho}^+(z_0)}^{1/2} \\
& \qquad + N \kappa^{-\gamma_1} \Big[\big(|\tilde F|^2\big)_{Q_{14 \rho}^+(z_0)}^{1/2} + \big(|\tilde{g}|^2\big)_{Q_{14 \rho}^+(z_0)}^{1/2}\Big],
\end{split}
\]
where $\tilde g =x_d^{1-\alpha} |\tilde f_1| + \lambda^{-1/2} x_d^{-\alpha/2}|\tilde f_2|$ and $N = N(d, \nu, \alpha) >0$.
Now, by H\"older's inequality,
\[
\begin{split}
& \big(|\tilde F|^2\big)_{Q_{14 \rho}^+(z_0)}^{1/2} \leq \big(|F|^2\big)_{Q_{14 \rho}^+(z_0)}^{1/2} + \left( \fint_{Q_{14 \rho}^+(z_0)} | a_{ij} -[a_{ij}]_{14\rho, z'_0}(x_d)|^2 |Du|^2\, dz \right)^{1/2} \\
& \leq \big(|F|^2\big)_{Q_{14 \rho}^+(z_0)}^{1/2}  + \big(|Du|^q\big)_{Q_{14 \rho}^+(z_0)}^{1/q} \left( \fint_{Q_{14 \rho}^+(z_0)} \big|a_{ij} -[a_{ij}]_{14\rho, z'_0}(x_d)\big|^{\frac{2q}{q-2}}\, dz \right)^{1/2-1/q}.
\end{split}
\]
Then it follows from the boundedness of $(a_{ij})$ in \eqref{ellipticity-cond} and \textup{Assumption} \ref{mean-osc} \textup{($\rho_0, \delta$)} that
\[
 \big(|\tilde F|^2\big)_{Q_{14 \rho}^+(z_0)}^{1/2} \leq \big(|F|^2\big)_{Q_{14 \rho}^+(z_0)}^{1/2}  + N(\nu, q)  \delta^{1/2-1/q} \big(|Du|^q\big)_{Q_{14 \rho}^+(z_0)}^{1/q}.
\]
Similarly, with the condition \eqref{c-mu.cond}, we also have
\[
\begin{split}
& \big(|\tilde g|^2\big)_{Q_{14 \rho}^+(z_0)}^{1/2} \leq \big(|g|^2\big)_{Q_{14 \rho}^+(z_0)}^{1/2} + \lambda^{1/2} \left(\fint_{Q_{14\rho}^+(z_0)} \big|[c_0]_{14\rho, z'_0}(x_d) - c_0\big|^2 \big|x_d^{-\alpha/2}u\big|^2 \right)^{1/2} \\
& \leq \big(|g|^2\big)_{Q_{14 \rho}^+(z_0)}^{1/2} + N(\nu, q) \delta^{1/2-1/q} \lambda^{1/2} \big(|x_d^{-\alpha/2}u|^q\big)_{Q_{14 \rho}^+(z_0)}^{1/q}.
\end{split}
\]
In conclusion, we obtain
\[
\begin{split}
&  \big(|\cU - (\cU)_{Q_{\kappa \rho}^+(z_0)}| \big)_{Q_{\kappa \rho}^+(z_0)} \\
& \leq N \Big[ \kappa^{\gamma_0} \big(|\cU|^2\big)_{Q_{14 \rho}^+(z_0)}^{1/2} + \kappa^{-\gamma_1}\delta^{1/2-1/q}  \Big(\big(|\cU|^q\big)_{Q_{14 \rho}^+(z_0)}^{1/q}
+\big(|v|^q\big)_{Q_{14 \rho}^+(z_0)}^{1/q}\Big) \Big]
 \\
& \quad  + N \kappa^{-\gamma_1} \Big[ \big(|F|^2\big)_{Q_{14 \rho}^+(z_0)}^{1/2} + \big(|g|^2\big)_{Q_{14 \rho}^+(z_0)}^{1/2}  \Big].
\end{split}
\]
From this, H\"{o}lder's inequality as $q>2$, and $|\cU| \leq N |Du|$, the mean oscillation estimates of $\cU$ is proved.
The mean oscillation estimate for $v = \lambda^{1/2} x_d^{-\alpha/2}u$ can be obtained similarly.
The proof of the lemma is completed.
\end{proof}

The next result  gives an oscillation estimate of solutions to \eqref{eq6.12}, each of which is supported in a small time interval.

\begin{lemma}
            \label{small-spt-lemma}
Let $\nu \in (0,1)$, $\alpha \in (0,2)$, $\rho_0, \delta>0$ be fixed numbers, and assume that  \eqref{ellipticity-cond}, \eqref{c-mu.cond}, and \textup{Assumption} \ref{mean-osc} \textup{($\rho_0, \delta$)} are satisfied.
Assume also that $F \in L_{2, \textup{loc}}(\Omega_T)$ and $f = f_1 + f_2$ such that  $g = x_d^{1-\alpha}|f_1| + \lambda^{-1/2} x_d^{-\alpha/2} |f_2| \in L_{2, \textup{loc}}(\Omega_T)$.  Assume further that  $u \in \sH_{q, \textup{loc}}^1(\Omega_T)$ is a weak solution to \eqref{eq6.12} with $q \in (2,\infty)$, and $\textup{spt}(u) \subset  (t_1 - (\rho_0 \rho_1)^{2-\alpha}, t_1 + (\rho_0 \rho_1)^{2-\alpha})$ for some  $t_1 \in \bR$ and $\rho_1>0$.
Then,
\[
\begin{split}
&  \big(|\cU - (\cU)_{Q_{\kappa \rho}^+(z_0)}| \big)_{Q_{\kappa \rho}^+(z_0)}  +\big(|v - (v)_{Q_{\kappa \rho}^+(z_0)}| \big)_{Q_{\kappa \rho}^+(z_0)} \\
& \leq N \big[ \kappa^{\gamma_0} + \kappa^{-\gamma_1}\delta^{1/2-1/q} + \kappa^{-2\gamma_1}\rho_1^{(1-1/q)(2-\alpha)}  \big] \big[ \big(|v|^q\big)_{Q_{14 \rho}^+(z_0)}^{1/q} +  \big(|Du|^q\big)_{Q_{14 \rho}^+(z_0)}^{1/q} \big]
 \\
& \quad  + N \kappa^{-\gamma_1} \Big[ \big(|F|^2\big)_{Q_{14 \rho}^+(z_0)}^{1/2} + \big(|g|^2\big)_{Q_{14 \rho}^+(z_0)}^{1/2}  \Big]
\end{split}
\]
for every $z_0 \in \overline{\Omega}_T$, $\rho >0$, and $\kappa \in (0,1)$, where  $N = N(\nu, \alpha, d, q)>0$ and  $\cU = (D_{x'}u, U)$ with $U = [a_{dj}]_{14\rho, z'_0}(x_d) D_ju$, and $v = \lambda^{1/2} x_d^{-\alpha/2} u$.
\end{lemma}

\begin{proof}
Note that if $\rho < \rho_0/14$,  the assertion of the lemma follows directly from Lemma \ref{oss-lem-vmo}.
It then remains to consider the case $\rho \geq \rho_0/14$.  We write $\Gamma = (t_1 - (\rho_0 \rho_1)^{2-\alpha}, t_1 + (\rho_0 \rho_1)^{2-\alpha})$.
It follows from \eqref{compared-Q}, the triangle inequality, and H\"{o}lder's inequality  that
\[
\begin{split}
&  \big(|\cU - (\cU)_{Q_{\kappa \rho}^+(z_0)}| \big)_{Q_{\kappa \rho}^+(z_0)}
\le 2\big(|\cU| \big)_{Q_{\kappa \rho}^+(z_0)}\\
& \leq N(d)\kappa^{-2\gamma_1} \left(\fint_{Q_{14\rho}^+(z_0)} |\cU|^q\, dz \right)^{1/q} \left(\fint_{Q_{14\rho}^+(z_0)} \chi_{\Gamma}(z)\, dz \right)^{1-1/q} \\
&  \leq N \kappa^{-2\gamma_1} \left(\frac{\rho_0\rho_1}{\rho}\right)^{(1-1/q)(2-\alpha)} \big(|\cU|^q\big)_{Q_{14\rho}^+(z_0)}^{1/q} \\
& \leq N \kappa^{-2\gamma_1}\rho_1^{(1-1/q)(2-\alpha)}\big(|\cU|^q\big)_{Q_{14\rho}^+(z_0)}^{1/q}.
\end{split}
\]
Therefore, the oscillation estimate for $\cU$ follows.
The oscillation estimate for $v$ can be proved similarly.
The proof of the lemma is completed.
\end{proof}

We now give a corollary of Lemma \ref{small-spt-lemma}, which proves the a priori estimate \eqref{main-est-0508} when $p>2$ and $u$ has a small support in time variable.

\begin{corollary} \label{small-support-lemma}
Let $\nu, \rho_0 \in (0,1)$, $\alpha \in (0,2)$, and $p \in (2, \infty)$.
There exist sufficiently small numbers $\delta = \delta(d, \nu, \alpha, p) >0$ and $\rho_1 = \rho_1 (d, \nu, \alpha, p) >0$ such that the following assertions hold.
Suppose that  \eqref{ellipticity-cond}, \eqref{c-mu.cond}, and  \textup{Assumption} \ref{mean-osc} \textup{($\rho_0, \delta$)} are satisfied, and suppose that $F \in L_p(\Omega_T)^d$   and $f = f_1 + f_2$ such that $g= x_d^{1-\alpha} |f_1| + \lambda^{-1/2}x_d^{-\alpha/2}|f_2| \in L_p(\Omega_T)$ with $\lambda>0$.
Then if  $u \in \sH^{1}_{p}(\Omega_T)$ is weak solution of \eqref{eq6.12} satisfying $\textup{spt}(u) \subset (t_1 - (\rho_1 \rho_0)^{2-\alpha}, t_1 +(\rho_1 \rho_0)^{2-\alpha})$ for some $t_1 \in \bR$, we have
\begin{equation} \label{lemma-est-0508-1}
\begin{split}
& \|Du\|_{L_p(\Omega_T)} +   \sqrt{\lambda} \| x_d^{-\alpha /2} u\|_{L_p(\Omega_T)}  \leq  N \Big[\|F\|_{L_p(\Omega_T)}  +   \|g\|_{L_p(\Omega_T)} \Big],
 \end{split}
\end{equation}
where $N = N(\nu, d, \alpha, p)>0$.
\end{corollary}

\begin{proof}
Let $q \in (2, p)$.
Recall that $|\cU|$ is comparable to $Du$.
By the mean oscillation estimates in Lemma \ref{small-spt-lemma},   we  follow the standard argument using the Fefferman-Stein sharp function theorem and the Hardy-Littlewood maximal function theorem (see, for instance,  \cite[Sec. 3.1-3.2]{Krylov} and \cite[Corollary 2.6, 2.7, and Sec. 7]{MR3812104}) to obtain
\[
\begin{split}
& \|Du\|_{L_p(\Omega)} + \lambda^{1/2} \|x_d^{-\alpha/2}u\|_{L_p(\Omega)} \\
& \leq N \Big[ \kappa^{\gamma_0} + \kappa^{-\gamma_1}\delta^{1/2-1/q} + \kappa^{-2\gamma_1}\rho_1^{(1-1/q)(2-\alpha)}  \Big]  \Big[\|Du\|_{L_p(\Omega)} + \lambda^{1/2} \|x_d^{-\alpha/2}u\|_{L_p(\Omega)}   \Big] \\
& \qquad + N \kappa^{-\gamma_1} \Big[\|F\|_{L_p(\Omega_T)} + \|g\|_{L_p(\Omega_T)}  \Big],
\end{split}
\]
where $N = N(\nu, d, p,\alpha)>0$.
We choose sufficiently small $\kappa$, then sufficiently small $\delta$ and $\rho_1$ so that
\[ N \Big[ \kappa^{\gamma_0} + \kappa^{-\gamma_1} \delta^{1/2-1/q} + \kappa^{-2\gamma_1}\rho_1^{(1-1/q)(2-\alpha)}  \Big] < 1/2.
\]
From this, \eqref{lemma-est-0508-1} follows.
\end{proof}

In the next lemma, we  prove the a priori estimate \eqref{main-est-0508} with $p \in (1,\infty)$ and no restriction on the support of solution $u$.

\begin{lemma} \label{apri-est.lemma}
Let $\nu, \rho_0 \in (0,1), \alpha \in (0,2)$ and $p \in (1, \infty)$. There exist a sufficiently small number $\delta = \delta(d, \nu, \alpha, p) >0$ and a sufficiently large number $\lambda_0 = \lambda_0 (d, \nu, \alpha, p) >0$ such that the following assertions hold.
Suppose that  \eqref{ellipticity-cond}, \eqref{c-mu.cond} and \textup{Assumption} \ref{mean-osc} \textup{($\rho_0, \delta$)} hold, $\lambda \geq \lambda_0 \rho^{\alpha-2}_{0}$, $F \in L_p(\Omega_T)^d$,   and $f = f_1 + f_2$ such that $g= x_d^{1-\alpha} |f_1| + \lambda^{-1/2}x_d^{-\alpha/2}|f_2| \in L_p(\Omega_T)$.
Then if $u \in \sH^{1}_{p}(\Omega_T)$ is weak solution of \eqref{eq6.12}, we have
\begin{equation*}
\begin{split}
& \|Du\|_{L_p(\Omega_T)} +   \sqrt{\lambda} \| x_d^{-\alpha /2} u\|_{L_p(\Omega_T)}  \leq  N \Big[\|F\|_{L_p(\Omega_T)}  +   \|g\|_{L_p(\Omega_T)} \Big],
 \end{split}
\end{equation*}
where $N = N(\nu, d, \alpha, p)>0$.
\end{lemma}

\begin{proof}
By Theorem \ref{L-2.theorem}, the assertion of the lemma holds when $p=2$.
It then remains to consider the cases when  $p \in (2,\infty)$ and $p \in (1,2)$.

\medskip

\noindent
{\bf Case 1:} $p \in (2,\infty)$.
We only need to remove the restriction on the support of the solution $u$ assumed in Corollary \ref{small-support-lemma}.
We use a partition of unity argument in the time variable.
Let $\delta>0$ and $\rho_1>0$ be as in Corollary \ref{small-support-lemma} and let
$$
\xi=\xi(t) \in C_0^\infty( -(\rho_0\rho_1)^{2-\alpha}, (\rho_0\rho_1)^{2-\alpha})
$$
be a  non-negative cut-off function satisfying
\begin{equation} \label{xi-0702}
\int_{\bR} \xi(s)^p\, ds =1 \quad \text{and} \quad  \int_{\bR}|\xi'(s)|^p\,ds \leq \frac{N}{(\rho_0\rho_1)^{p(2-\alpha)}}.
\end{equation}
For fixed $s \in (-\infty,  \infty)$, let $u^{(s)}(z) = u(z) \xi(t-s)$ for $z = (t, x) \in \Omega_T$.
We see that $u^{(s)} \in \sH_p^1(\Omega_T)$ is a weak solution of
\[
u^{(s)}_t + \lambda c_0(z) u^{(s)} - \mu(x_d) D_i\big(a_{ij} D_j u^{(s)} - F^{(s)}_{i}\big)  =  f^{(s)}
\]
in $\Omega_T$ with the boundary condition $u^{(s)} =0$ on $\{x_d =0\}$, where
\[
F^{(s)}(z) = \xi(t-s) F(z), \quad f^{(s)}(z)   = \xi(t-s) f(z)  +  \xi'(t-s) u(z).
\]
As $\text{spt}(u^{(s)}) \subset (s -(\rho_0\rho_1)^{2-\alpha}, s+ (\rho_0\rho_1)^{2-\alpha}) \times \bR^{d}_{+}$, we apply Corollary \ref{small-support-lemma} to get
\[
\begin{split}
& \|Du^{(s)}\|_{L_p(\Omega_T)} + \sqrt{\lambda} \|x_d^{-\alpha/2}u^{(s)}\|_{L_p(\Omega_T)}  \\
& \leq N \Big( \|F^{(s)}\|_{L_p(\Omega_T)} +\|g^{(s)}\|_{L_p(\Omega_T)} + \lambda^{-1/2}\|x_d^{-\alpha/2}u\xi'(\cdot-s)\|_{L_p(\Omega_T}\big),
\end{split}
\]
where
\[g^{(s)}(z) = \big( x_d^{1-\alpha} |f_1(z)| + \lambda^{-1/2} x_d^{-\alpha/2}|f_2(z)| \big)  \xi(t-s), \quad z = (t, x', x_d) \in \Omega_T.
\]
Then, by integrating the $p$-power of this estimate with respect to $s$, we get
\begin{align}\notag
& \int_{\bR}\Big( \|Du^{(s)}\|_{L_p(\Omega_T)}^p + \lambda^{p/2} \|x_d^{-\alpha/2}u^{(s)}\|^p_{L_p(\Omega_T)}\Big)\, ds\\ \notag
&  \leq N\int_{\bR} \Big( \|F^{(s)}\|^p_{L_p(\Omega_T)} + \|g^{(s)}\|^p_{L_p(\Omega_T)} \\  \label{int-0515}
& \qquad + \lambda^{-1/2}\|x_d^{-\alpha/2}  u \xi'(\cdot-s)\|_{L_p(\Omega_T)}^p \Big)\, ds.
\end{align}
Now, by the Fubini theorem and \eqref{xi-0702}, it follows that
\[
\int_{\bR}\|Du^{(s)}\|_{L_p(\Omega_T)}^p\, ds = \int_{\Omega_T}\int_{\bR} |Du(z)|^p \xi^p(t-s)\, dsdz  = \|Du\|_{L_p(\Omega_T)}^p,
\]
and similarly
\[
\begin{split}
& \int_{\bR}\|x_d^{-\alpha/2} u^{(s)}\|_{L_p(\Omega_T)}^p\, ds = \|x_d^{-\alpha /2} u\|_{L_p(\Omega_T)}^p, \\
& \int_{\bR}\|F^{(s)}\|_{L_p(\Omega_T)}^p\, ds = \|F\|_{L_p(\Omega_T)}^p, \quad \int_{\bR}\|g^{(s)}\|_{L_p(\Omega_T)}^p\, ds = \|g\|_{L_p(\Omega_T)}^p.
\end{split}
\]
Moreover,
\[
\int_{\bR}\|x_d^{-\alpha/2}  u \xi'(\cdot-s)\|_{L_p(\Omega_T}^pds \leq N \rho_0^{p(\alpha-2)} \|x_d^{-\alpha/2} u\|_{L_p(\Omega_T)}^p,
\]
where \eqref{xi-0702} is used and $N = N(d, \nu, \alpha, p)>0$.
Then, by combining the estimates we just derived, we infer from \eqref{int-0515} that
\[
\begin{split}
& \|Du\|_{L_p(\Omega_T)} + \sqrt{\lambda} \|x_d^{-\alpha/2}u\|_{L_p(\Omega_T)}\\
&  \leq N\Big(\|F \|_{L_p(\Omega_T)} + \|g\| _{L_p(\Omega_T)} + \rho_0^{\alpha-2}\lambda^{-1/2}\|x_d^{-\alpha/2}u\|_{L_p(\Omega_T)}\Big)
\end{split}
\]
with $N=N(d, \nu, \alpha, p)$.
Now we choose $\lambda_0 = 2N$.
Then, with $\lambda \geq \lambda_0 \rho_0^{\alpha-2}$, we have $N\rho_0^{\alpha-2}\lambda^{-1/2} \leq \sqrt\lambda/2$, and consequently
\[
\begin{split}
& \|Du \|_{L_p(\Omega_T)}
+ \sqrt{\lambda} \|x_d^{-\alpha/2}u \|_{L_p(\Omega_T)}\\
&  \leq   N \|F \|_{L_p(\Omega_T)} +N \|g \| _{L_p(\Omega_T)}+\frac{\sqrt{\lambda}}{2} \|x_d^{-\alpha/2}u\|_{L_p(\Omega_T)}.
\end{split}
\]
This estimate yields \eqref{main-est-0508}.

\medskip

\noindent
{\bf Case 2:} $p \in (1,2)$.
We apply the duality argument.
This can be done exactly the same as that of the proof of Theorem \ref{sim-eqn-thrm}.
We skip the details.
\end{proof}

\begin{proof}[Proof of Theorem \ref{main-thrm}]
Let $\delta$ and $\lambda_0$ be defined in Lemma \ref{apri-est.lemma}.
Then from Lemma \ref{apri-est.lemma}, we see that \eqref{main-est-0508} holds for every weak solution $u$ of \eqref{eq6.12}.
The existence of the solution $u \in \sH_p^1(\Omega_T)$ can be obtained by the method of continuity using the solvability in Theorem \ref{sim-eqn-thrm}.
The proof of the theorem is completed.
\end{proof}

In order to prove Theorem \ref{thrm-2}, we need an additional lemma, which is a generalization of Proposition \ref{prop3}.

\begin{lemma}
                        \label{lem5.2}
Let $p_0\in (1,2)$ and suppose that \eqref{c-mu.cond},  \eqref{a-c.cond}, and \eqref{elli-cond} are satisfied in $Q_1^+$. If $u\in \sH^1_{p_0}(Q_1^+)$ is a weak solution of \eqref{eq6.12h} in $Q_1^+$, then we have
\begin{align}
                        \label{eq8.10c}
&\|u\|_{C^{1,1}(Q_{1/2}^+)}+
\|D_{x'}u\|_{C^{1,1}(Q_{1/2}^+)}+\|U\|_{C^{1,\gamma}(Q_{1/2}^+)}\notag\\
&\quad +\sqrt\lambda\|u x_d^{-\alpha/2}\|_{C^{1,1-\alpha/2}(Q_{1/2}^+)}\le N \|Du\|_{L_{p_0}(Q_1^+)},
\end{align}
where $N=N(d,\nu,\alpha,p_0)>0$, $\gamma=\min\{2-\alpha,1\}$, and $U(z)=\overline{a}_{dj}(x_d)D_j u(z)$ for $z = (z', x_d) \in Q_1^+$.
\end{lemma}

\begin{proof}
Let $\eta_1\in C_0^\infty((0,1/4))$ and $\eta_2\in C_0^\infty(B_1')$ be nonnegative functions with unit integral.
For $\varepsilon>0$, let
$$
u^{(\varepsilon)}(t,x)=\int_{\bR^d} u(t-\varepsilon^2 s,x'-\varepsilon y',x_d)\eta_1(s)\eta_2(y')\,dy'ds
$$
be the mollification of $u$ with respect to $t$ and $x'$.
Then we have
$\partial_t^kD_{x'}^jD_d^{l}u^{(\varepsilon)}\in L_{p_0}(Q^+_{3/4})$ for any $k,l\ge 0$, $l=0,1$, and any sufficiently small $\varepsilon>0$.
By the Sobolev embedding theorem, we get $u^{(\varepsilon)}, D_{x'}u^{(\varepsilon)}\in L_\infty(Q^+_{3/4})$. Following the proof of \eqref{eq12.20}, we also have $U^{(\varepsilon)}:=\overline{a}_{dj}(x_d)D_j u^{(\varepsilon)}\in L_\infty(Q^+_{3/4})$. In particular, we get $Du^{(\varepsilon)}\in L_2 (Q^+_{3/4})$, which also implies that $u^{(\varepsilon)}x_d^{-\alpha/2}\in L_2 (Q^+_{3/4})$ by using Hardy's inequality. Therefore, $u^{(\varepsilon)}\in \sH^1_{2}(Q_{3/4}^+)$. Now by Proposition \ref{prop3}, we have
\begin{align*}
&\|u^{(\varepsilon)}\|_{C^{1,1}(Q_{1/2}^+)}+
\|D_{x'}u^{(\varepsilon)}\|_{C^{1,1}(Q_{1/2}^+)}+\|U^{(\varepsilon)}\|_{C^{1,\gamma}(Q_{1/2}^+)}\\
&\quad +\sqrt\lambda\|u^{(\varepsilon)} x_d^{-\alpha/2}\|_{C^{1,1-\alpha/2}(Q_{1/2}^+)}\le N \|Du^{(\varepsilon)}\|_{L_{2}(Q_{2/3}^+)}.
\end{align*}
By using a standard iteration argument, we obtain
\begin{align*}
&\|u^{(\varepsilon)}\|_{C^{1,1}(Q_{1/2}^+)}+
\|D_{x'}u^{(\varepsilon)}\|_{C^{1,1}(Q_{1/2}^+)}+\|U^{(\varepsilon)}\|_{C^{1,\gamma}(Q_{1/2}^+)}\\
&\quad +\sqrt\lambda\|u^{(\varepsilon)} x_d^{-\alpha/2}\|_{C^{1,1-\alpha/2}(Q_{1/2}^+)}\le N \|Du^{(\varepsilon)}\|_{L_{p_0}(Q_{3/4}^+)},
\end{align*}
which implies \eqref{eq8.10c} after passing to the limit as $\varepsilon\to 0$.
The lemma is proved.
\end{proof}

We are now ready to give the proof of Theorem \ref{thrm-2}.

\begin{proof}[Proof of Theorem \ref{thrm-2}]
We give a sketch of the proof.
By using Theorem \ref{sim-eqn-thrm} and Lemma \ref{lem5.2}, we have the following mean oscillation estimate analogous to the one in Lemma \ref{oss-lem-vmo}.
Let $1<p_0<p_1<2$, $\lambda>0$, and $u \in \sH_{p_1, \textup{loc}}^1(\Omega_T)$ be a weak solution of \eqref{eq6.12} with $F \in L_{p_0, \textup{loc}}(\Omega_T)$ and $f = f_1 + f_2$ such that $g = x_d^{1-\alpha}|f_1| + \lambda^{-1/2} x_d^{-\alpha/2}|f_2| \in L_{p_0, \textup{loc}}(\Omega_T)$.
Then there is a constant $N = N(\nu, \alpha, d, p_1,p_2)>0$ such that
\[
\begin{split}
&\big(|\cU - (\cU)_{Q_{\kappa \rho}^+(z_0)}| \big)_{Q_{\kappa \rho}^+(z_0)} +
\big(|v - (v)_{Q_{\kappa \rho}^+(z_0)}| \big)_{Q_{\kappa \rho}^+(z_0)} \\
& \leq N \big( \kappa^{\gamma_0} + \kappa^{-\gamma_1} \delta^{1/p_0-1/p_1} \big)\Big[ \big(|v|^{p_1}\big)_{Q_{14 \rho}^+(z_0)}^{1/p_1}
+ \big(|Du|^{p_1}\big)_{Q_{14 \rho}^+(z_0)}^{1/{p_1}} \Big]
 \\
& \quad  + N \kappa^{-\gamma_1} \Big[ \big(|F|^{p_0}\big)_{Q_{14 \rho}^+(z_0)}^{1/{p_0}} + \big(|g|^{p_0}\big)_{Q_{14 \rho}^+(z_0)}^{1/{p_0}}  \Big]
\end{split}
\]
for every $z_0 \in \overline{\Omega}_T$, $\rho \in (0, \rho_0/14)$, and $\kappa \in (0,1)$, where $\cU$ and $v$ are defined as in Lemma \ref{oss-lem-vmo}, $\gamma_0 = \min\{1, 2-\alpha\}/ 4$, and $\gamma_1 = (d+2-\alpha)/p_0$. With this mean oscillation estimate in hand, we can derive the weighted a priori estimate \eqref{main-est-0508b} as in the proof of Theorem \ref{main-thrm} by using the weighted Fefferman-Stein sharp function theorem and the Hardy-Littlewood maximal function theorem (see, for instance, \cite[Corollary 2.6, 2.7, and Sec. 7]{MR3812104}) as well as a partition of unity in the time variable as in the proof of Lemma \ref{apri-est.lemma}. Finally, to show the solvability, we use the solvability in unweighted Sobolev spaces in Theorem \ref{main-thrm} and follow the argument in \cite[Sec. 8]{MR3812104}. The theorem is proved.
\end{proof}

\end{document}